\pdfoutput=1
\documentclass{wsbart/wsbart}

\newmuskip\tinymuskip
\renewcommand{\`}{\mskip-\tinymuskip}

\usepackage[utf8]{inputenc}
\usepackage[T1]{fontenc}
\usepackage{lmodern}

\usepackage[maxnames=99,style=trad-plain,backend=biber,isbn=false,
uniquename=init]
{biblatex}
\DeclareRedundantLanguages{English}{english}

\usepackage{amsmath,amssymb}
\usepackage{amsthm}
\usepackage{bm,dsfont,stmaryrd}
\SetSymbolFont{stmry}{bold}{U}{stmry}{m}{n}
\renewcommand{\mathbb}{\mathds}
\usepackage{mathtools}
\usepackage{accents}
\usepackage{IEEEtrantools}

\newcommand{\R}{\mathbb R}
\newcommand{\T}{\mathbb T}
\newcommand{\X}{\mathcal X}
\newcommand{\x}{\bm{x}}
\newcommand{\vect}[1]{\bm{#1}}

\newcommand{\dd}{\mathop{}\!\mathrm{d}}
\DeclareMathOperator{\Law}{Law}

\DeclareMathOperator{\Expect}{\mathbb{E}}
\DeclareMathOperator{\Var}{Var}
\DeclareMathOperator{\Osc}{Osc}
\DeclareMathOperator{\sgn}{sgn}
\newcommand{\1}{\mathbb{1}}

\newcommand{\adjustbin}{\negmedspace{}}

\newcommand{\Jc}{J_\textnormal{c}}
\newcommand{\TTE}[1]{\textnormal{T}_{\`#1}}

\newtheorem{thm}{Theorem}
\newtheorem*{thm*}{Theorem}
\newtheorem{lem}[thm]{Lemma}
\newtheorem{cor}[thm]{Corollary}
\newtheorem{prop}[thm]{Proposition}
\newtheorem*{prop*}{Proposition}
\theoremstyle{definition}

\newtheorem*{assu*}{Assumption}
\theoremstyle{remark}
\newtheorem{rem}{Remark}
\newtheorem{exm}{Example}

\title{Size of chaos for Gibbs measures\\of mean field interacting diffusions}
\author[1]{Zhenjie Ren}
\author[2]{Songbo Wang}
\affil[1]{LaMME, Université Évry Paris-Saclay, Évry, France}
\affil[2]{LJAD, Université Côte d'Azur, Nice, France}

\begin{document}

\maketitle

\begin{abstract}
We investigate Gibbs measures for diffusive particles
interacting through a two-body mean field energy.
By identifying a gradient structure for the
conditional law, we derive sharp bounds on the size of chaos,
providing a quantitative characterization of particle independence.
To handle interaction forces that are unbounded at infinity,
we study the concentration of measure phenomenon for Gibbs measures
via a defective Talagrand inequality, which may hold independent interest.
Our approach provides a unified framework
for both the flat semi-convex and displacement convex cases.
Additionally, we establish sharp chaos bounds
for the quartic Curie--Weiss model in the sub-critical regime,
demonstrating the generality of this method.
\end{abstract}

\section{Introduction}

Let $d$ be an integer $\geqslant 1$
and let $\X$ be the $d$-dimensional torus $\T^d$
or the Euclidean space $\R^d$.
Let $V \colon \X \to \R$ and $W \colon \X \times \X \to \R$ be $\mathcal C^1$
and suppose that $W$ is symmetric.
We call $V$, $W$
the \emph{confinement} and \emph{interaction} potentials respectively.
The central object of this paper is the \emph{$N$-particle
Gibbs measure} $m^N_*$, defined by
\[
m^N_*(\dd x^1\cdots\dd x^N)
= \frac{\exp \bigl( - \sum_{i \in [N]} V(x^i)
- \frac 1{2N} \sum_{i,j \in [N]} W(x^i,x^j) \bigr)
\dd x^1\cdots\dd x^N}
{\int_{\X^N}\exp \bigl( - \sum_{i \in [N]} V(y^i)
- \frac 1{2N} \sum_{i,j \in [N]} W(y^i,y^j) \bigr)
\dd y^1\cdots\dd y^N},
\]
where the $1/N$ factor represents the \emph{mean field} scaling
for interactions between the $N$ particles.
In the case $\X = \R^d$, we assume in addition that
the term in the exponent is upper bounded
and that the measure $m^N_*$ is well defined for $N \geqslant 1$.
We also assume throughout the paper that
there exists a unique \emph{mean field Gibbs measure} $m_* \in \mathcal P(\X)$
solving the fixed point equation
\begin{equation}
\label{eq:m*-fixed-point}
m_* = \Pi [m_*],
\end{equation}
where the mapping $\Pi$ between probability measures is defined by
\begin{equation}
\label{eq:def-Pi}
\Pi[m](\dd x) =
\frac{\exp\bigl( -V(x) - \langle W(x,\cdot), m\rangle\bigr)\dd x}
{\int_{\X} \exp\bigl( -V(y) - \langle W(y,\cdot), m\rangle \dd y}.
\end{equation}
Here $\langle \cdot, \cdot\rangle$ denotes the integral of a function
with respect to a measure.
For $\X = \R^d$, we suppose additionally that the measure $\Pi[m]$
is well defined for all $m \in \mathcal P_2(\R^d)$
and $m_* \in \mathcal P_2(\R^d)$.

By direct computation,
the Gibbs measure $m^N_*$ minimizes
the \emph{$N$-particle free energy functional}
$\mathcal F^N \colon \mathcal P_2(\mathcal X^N) \to (-\infty,+\infty]$,
defined by
\[
\mathcal F^N(\nu^N)
\coloneqq N \Expect_{\bm{X} \sim \nu^N} \biggl[
\langle V, \mu_{\bm{X}}\rangle + \frac 12 \langle W,
\mu_{\bm{X}}^{\otimes 2}\rangle\biggr] + H(\nu^N).
\]
Here $\mu_{\bm{X}}$ denotes the $N$-particle empirical measure for
the $N$-particle configuration $\bm{X} \coloneqq (X^1,\ldots,X^N)$:
\[
\mu_{\bm{X}} \coloneqq \frac 1N \sum_{i \in [N]} \delta_{X^i}
\]
and $H(\nu^N)$ represents the differential entropy, given by
\[
H(\nu^N) \coloneqq \int_{\X^N} \nu^N(\vect x)
\log \nu^N(\vect x) \dd \vect x,
\]
where $\nu^N(\cdot)$ denotes the density function of the measure $\nu^N$.
Similarly, the one-particle measure $m_*$ is expected to minimize
the \emph{mean field energy functional}
$\mathcal F \colon \mathcal P_2(\X) \to (-\infty, +\infty]$, defined as
\[
\mathcal F(m) = \langle V, m\rangle + \frac 12 \langle W, m^{\otimes 2}\rangle
+ H(m).
\]
Indeed, the fixed point equation \eqref{eq:m*-fixed-point}
is precisely the first-order condition for minimizing the free energy
$\mathcal F$.

Through formal computations, it can be verified that
$m^N_*$ is the invariant measure
of the $N$-particle \emph{interacting diffusion} process
\[
\dd X^i_t = - \nabla V(X^i_t) \dd t
- \frac 1N \sum_{j \in [N]} \nabla_1 W(X^i_t, X^j_t) \dd t
+ \sqrt 2 \dd B^i_t,\qquad i \in [N].
\]
The mean field Gibbs measure $m_*$ can also be shown to be invariant
under the \emph{non-linear diffusion} process
\begin{equation}
\label{eq:mfl}
\dd X_t = - \nabla V(X_t) \dd t
- \langle \nabla_1 W(X_t, \cdot), m_t\rangle \dd t
+ \sqrt 2 \dd B_t, \qquad m_t = \Law(X_t).
\end{equation}
In these equations, $B^i_t$ and $B_t$ are independent
$d$-dimensional Brownian motions.
For more details on this classical link between the Gibbs measure
and the dynamical process, we refer readers to our previous work \cite{ulpoc}.

The main purpose of this paper is to investigate
the \emph{quantitative chaotic behavior} of the sequence $m^N_*$ towards $m_*$.
To establish our framework, let $k \geqslant 1$ and
for all $N \geqslant k$, we define the marginal distributions
\[
m^{N,k}_*(\cdot)
= \int_{\X^{N-k}} m^N_*(\cdot,\dd x^{k+1}\cdots\dd x^N).
\]
We say the sequence $(m^N_*)_{N \geqslant 1}$ is
\emph{$m_*$-chaotic} if for all $k \geqslant 1$,
the weak convergence
\[
m^{N,k}_* \rightharpoonup m_*^{\otimes k}
\]
holds when $N \to \infty$.
This classical notion of chaos, introduced by \textcite{KacFoundations},
is fundamental in the study of large particle systems.
Demonstrating its propagation over time
is a key objective of kinetic theory,
as it provides a physical basis
for partial differential equations describing the system's evolution.
Notable contributions to qualitative chaos include
Sznitman's martingale approach \cite{SznitmanBoltzmannHomogenes}
to the Boltzmann--Kac particle system
and the analysis by \textcite{FHMVortex}
on the 2D viscous vortex model.
See \cite[Section~4.2]{ChaintronDiezPOC1}
for a detailed review of these results.

Our aim is to go beyond this qualitative convergence
and instead \emph{quantify} the disparity between
the two probability measures $m^{N,k}_*$ and $m_*^{\otimes k}$
for arbitrary $k$ and $N$.
This quantitative problem is referred to as the \emph{size of chaos}
in the recent literatures \cite{PPSChaos,DuerinckxChaos}.
Following the approach of \textcite{LackerQuantitative},
we consider the \emph{local} relative entropy:
\[
H(m^{N,k}_* | m_*^{\otimes k})
\coloneqq \int_{\X^k}
\log \frac{m^{N,k}_*(\x^{[k]})}{m_*^{\otimes k}(\x^{[k]})}
m^{N,k}_* (\dd\x^{[k]}),
\]
where the measures are identified with their
densities and $\x^{[k]}$ denotes the $k$-tuple
$(x^1,\ldots,x^k)$.
A key advantage of relative entropy is its linear scaling
with the number of particles.
More specifically, it satisfies the subadditivity property:
\[
H(m^{N,k}_* | m_*^{\otimes k})
\leqslant \frac{1}{\lfloor N/k\rfloor}
H(m^{N}_* | m_*^{\otimes N}),
\]
$\lfloor\cdot\rfloor$ denoting the integer part.
In scenarios involving weak interaction potential $W$,
the \emph{global} entropy typically exhibits the asymptotic behavior
\[
H(m^{N}_* | m_*^{\otimes N}) = O(1)
\]
when $N\to\infty$.
Consequently, by subadditivity,
\[
H(m^{N,k}_* | m_*^{\otimes k}) = O\biggl( \frac kN \biggr).
\]
This result not only provides a quantitative bound on the size of chaos but also
establishes $m_*$-chaoticity of the sequence $m^N_*$, as $k/N \to 0$ when
$N \to \infty$ for fixed $k$.
However, this global-to-local strategy breaks down for stronger metrics
such as $L^2$ or $L^p$ norms, which do not exhibit the linear scaling property.
In those cases, a purely local approach based on the BBGKY hierarchy
is necessary to establish chaoticity,
as exemplified in the recent works of Bresch--Jabin--Soler \cite{BJSNewApproach}
and Hess-Childs--Rowan \cite{HCRHigher}.
Regarding hierarchy-based methods for establishing chaoticity,
we also highlight the recent contributions of Bresch--Duerinckx--Jabin
\cite{BDJDuality} and Jabin--Poyato--Soler \cite{JPSMFLNonExchangeable}.

While the global-local transition in the entropy method
offers conceptual and technical convenience,
the subadditivity fails to yield optimal bounds for the size of chaos,
even in elementary Gaussian scenarios.
Lacker has demonstrated the limitation both for
Gibbs measures \cite[Example~2.3]{LackerQuantitative}
and for stochastic processes \cite[Section~3]{LackerHier}.
Specifically,
the \emph{sharp chaos bound} for Gaussian variables is instead the following:
\[
H(m^{N,k}_* | m_*^{\otimes k}) = O\biggl( \frac {k^2}{N^2} \biggr),
\]
indicating that the passage to local entropy incurs a loss of exponents
when $k/N = o(1)$.
In \cite{LackerQuantitative,LackerHier}, Lacker develops techniques based on
the BBGKY hierarchy and log-Sobolev inequalities
to establish this sharp chaos bound for a wide class
of regular potentials $V$ and $W$,
demonstrating the genericity of the $O(k^2\!/N^2)$ bound.

Several subsequent works have extended
these results in the dynamical setting.
First, \textcite{LLFSharp} show that,
under a weak and regular interaction potential,
the size of chaos remains uniformly bounded in time.
\textcite{HanEnt} extends these results to divergence-free
and $L^p$ drift kernels for $p > d$,
albeit with strong restrictions on the initial data.
One of the authors \cite{slpoc} removes these restrictions,
proving sharp propagation of chaos for the broader $W^{-1,\infty}$ class
in high viscosity
and achieving time-uniform estimates in the divergence-free case.
\textcite{HCRHigher}, focusing on bounded drift kernels,
construct an impressive power series expansion for the $N$-particle evolution
in terms of $1/N$,
showing that the size of the $n$th perturbative order
is $O(k^{2n}\!/N^{2n})$.
More recently, \textcite{GGPSharp}
extend Lacker's approach to scenarios with variable diffusion coefficients.
Additionally, \textcite{LYZNonExchangeable}
explore particle systems interacting via a graph and uncover connections to
first-passage percolation,
further expanding the scope and applicability of sharp chaos bounds.

While the aforementioned efforts address the propagation of sharp chaos,
this paper is devoted to extending sharp chaos bounds for Gibbs measures
to settings not covered by Lacker's original work~\cite{LackerQuantitative}.
A key simplification, in constrast to the dynamical case,
arises from the explicit form of the densities of
Gibbs measures, which enables a more precise analysis of their chaotic
structure. Our approach does not appear to be directly applicable to
the dynamical setting, which is naturally more challenging. For further
discussion regarding the dynamical case, we refer readers to the end of the main
results section.

In \cite{LackerQuantitative}, Lacker performs a direct comparison between
the $k$-particle marginal $m^{N,k}_*$ and the tensorized measure
$m_*^{\otimes k}$.
This approach constructs a hierarchy of marginal Fisher informations:
\[
I(m^{N,k}_*|m_*^{\otimes k})
= \int_{\X^k} \biggl| \nabla \log \frac{m^{N,k}_*(\x^{[k]})}
{m_*^{\otimes k}(\x^{[k]})} \biggr|^2 m^{N,k}_*(\dd\x^{[k]}),
\qquad k \in [N],
\]
where controlling growth across levels requires weak interactions,
specifically the smallness of $W$.
To overcome this limitation, we propose a novel methodology based on a hierarchy
of conditional entropies, which we outline below.

Our approach centers on analyzing the hierarchy of entropies
associated with the conditional measure
\[
m^{N,k|[k-1]}_{*,\x^{[k-1]}}(x^k)
\coloneqq \frac{m^{N,k}_*(\x^{[k]})}{m_*^{N,k-1}(\x^{[k-1]})},
\]
which represents the law of the $k$th particle
conditioned on the configurations of the first $k-1$ particles.
The average conditional entropy of this measure
can be expressed using the chain rule as:
\[
\int_{\X^{k-1}}
H(m^{N,k|[k-1]}_{*,\x^{[k-1]}}|m_*) m^{N,k-1}_*(\dd\x^{[k-1]})
= H(m^{N,k}_*|m_*^{\otimes k})
- H(m^{N,k-1}_*|m_*^{\otimes (k-1)}).
\]
Thus, we consider a higher-order hierarchy than that employed
in Lacker's approach, which focuses on entropies
or Fisher informations associated with marginal distributions.
At equilibrium, this conditional law is characterized by two forces:
the interaction among the first $k$ particles,
encoded by the empirical measure
\[
\mu_{\x^{[k]}} \coloneqq \frac{1}{k} \sum_{i \in [k]} \delta_{x^i},
\]
and the interaction with the remaining $N-k$ particles,
encoded in the $(k+1)$th-level conditional measure
\[
m^{N,k+1|[k]}_{*,\x^{[k]}}.
\]
We hypothesize that these two conditional measures of neighboring levels
are closely aligned.
A direct calculation confirms that their average relative entropy corresponds
to the second-order entropy difference:
\begin{multline*}
\int_{\X^k}
H(m^{N,k+1|[k]}_{*,\x^{[k]}} | m^{N,k|[k-1]}_{*,\x^{[k-1]}})
m^{N,k}_*(\dd\x^{[k]}) \\
= H(m^{N,k+1}_*|m_*^{\otimes (k+1)})
- 2H(m^{N,k}_*|m_*^{\otimes k})
+ H(m^{N,k-1}_*|m_*^{\otimes (k-1)}).
\end{multline*}
This higher-order error term allows
for a more manageable hierarchical structure,
whose qualitative behavior is unaffected by the magnitude of $W$.
Furthermore, the comparison of the conditional laws
$m^{N,k+1|[k]}_{*,\vect x^{[k]}}$
and $m^{N,k|[k-1]}_{*,\vect x^{[k-1]}}$
introduces the non-linear Fisher information
\[
I(m^{N,k|[k-1]}_{*,\x^{[k-1]}} | \Pi[m^{N,k|[k-1]}_{*,\x^{[k-1]}}])
\]
as a dissipation term, which is directly connected
to the free energy landscape of the non-linear diffusion process \eqref{eq:mfl}.

This observation offers a non-perturbative approach to establishing sharp chaos
bounds that relies solely on the dissipation structure of the free energy,
rather than on the smallness of the interaction energy. Consequently, it enables
the proof of sharp chaotic estimates in many scenarios that Lacker's method
\cite[Theorem~2.2]{LackerQuantitative} fails to cover.
Notable cases include the flat convex framework studied by
\textcite{NWSConvexMFL}, as well as \textcite{ChizatMFL}.
Interesting examples within flat convexity encompass
mean field two-layer neural networks
(see the two above-cited works and also \cite{HRSS})
and repulsive Coulomb gases.
However, for the latter, our method is limited to dimension $1$ due to
its inability to handle singular interactions.
Additionally, we examine the classical displacement convexity case studied by
\textcite{CMVKinetic}, where our method overcomes the
sub-optimal regime observed in Lacker's reversed entropy approach
\cite[Theorem~2.8]{LackerQuantitative}.
To further illustrate the broad reach of our framework,
we study the celebrated quartic Curie--Weiss model and
establish sharp chaos bounds up to the critical point.

One technical challenge arises when comparing conditional laws at neighboring
levels, as it is necessary to verify a transport-cost inequality for these
distributions. For Lipschitz interaction forces, which appear in many of the
examples above, this verification reduces to establishing Gaussian concentration
for the relevant conditional distributions. We resolve this issue by proving a
dimension-free \emph{defective $\TTE2$ inequality} for the $N$-particle Gibbs
measure, under a coercivity condition on the free energy functional, which we
refer to as a non-linear $\TTE2$ inequality in the following.
The proof relies on a conditional decomposition of the associated
\emph{modulated free energy}, originally introduced by \textcite{BJWAttractive}
for diffusive particles with singular interaction,
and, as a by product, provides a large deviation estimate
of the Jabin--Z.~Wang type.
These results extend the classical tensorization property
of the $\TTE2$ inequality and may also be of independent interest.

As a final remark, our approach is solely based
on the mean field dissipation structure
and does not require detailed analysis of the particle system.
This method shares similarities with the work
of \textcite{DelarueTseUniformPOC},
where time-uniform propagation of chaos is established
by leveraging exclusively the stability of the mean field measure flow.
Unlike existing methods, our approach does not depend on uniform
log-Sobolev inequalities developed for particle systems,
such as Malrieu's results in the displacement convex case \cite{MalrieuLSI}.
For the flat convex case, we avoid relying on findings
from our previous works \cite{ulpoc,nulsi}
and their extensions by \textcite{MonmarcheULSIBeyond},
as well as the work of \textcite{CNZULSI}.
Similarly, for the Curie--Weiss model,
our method is independent of Bauerschmidt, Bodineau and Dagallier's
parallel findings on log-Sobolev inequalities~\cite{BBDCriterionFreeEnergy},
though their study provides valuable insights into the critical exponent.

\bigskip

The next section contains the main results of this paper
and the following sections are dedicated to proving these results
in the order they appear.

\section{Main results}

We now present the main results.
The first theorem establishes the framework for deriving sharp chaos bounds.

\begin{thm}
\label{thm:sharp}
Suppose that the following three conditions hold true.
\begin{itemize}
\item
(Log-Sobolev inequalities)
There exists $\rho > 0$ such that the non-linear and linear
log-Sobolev inequalities hold:
\[
\forall m \in \mathcal P(\X),
\qquad
2\rho H(m|m_*)
\leqslant
\begin{cases}
I(m|\Pi[m]),~\text{and} \\
I(m|m_*).
\end{cases}
\]
\item
(Transport inequalities)
There exists $\gamma \geqslant 0$ such that
for both $\mu = m^{N,k|[k-1]}_{*,\x^{[k-1]}}$ and $\mu = m_*$,
the probability measure $\mu$ satisfies the transport inequality
\[
\forall \nu \in \mathcal P(\X),\qquad
\lvert \langle \nabla_1 W(x^k, \cdot),
\nu - \mu \rangle\rvert^2
\leqslant \gamma
H(\nu | \mu )
\]
uniformly in $k \in [N]$, $\x^{[k]} \in \X^{k}$.
\item
(Square integrability)
There exists $M \geqslant 0$ such that
\begin{align*}
\int_{\X^2} \lvert \nabla_1 W(x,y)
- \langle \nabla_1 W(x,\cdot), m_*\rangle\rvert^2
m^{N,2}_*(\dd x\dd y) &\leqslant M, \\
\int_{\X} \lvert \nabla_1 W(x,x)
- \langle \nabla_1 W(x,\cdot), m_*\rangle\rvert^2
m^{N,1}_*(\dd x) &\leqslant M.
\end{align*}
\end{itemize}
Then for all $k \in [N]$,
\[
\int_{\X^{k-1}} H(m^{N,k|[k-1]}_{*,\x^{[k-1]}} | m_*)
m^{N,k-1}_*(\dd\x^{[k-1]})
\leqslant
\frac{36\bigl(1+\frac\gamma\rho\bigr)^3M}{\rho N^2}
\biggl( k + \frac{3\gamma}{\rho}\biggr),
\]
and consequently,
\[
H(m^{N,k}_* | m_*^{\otimes k}) \leqslant
\frac{18\bigl(1+\frac\gamma\rho\bigr)^3M}{\rho N^2}
\biggl( k^2+ \Bigl(1+\frac{6\gamma}{\rho}\Bigr)k\biggr).
\]
\end{thm}

The lower bound on $I(m|\Pi[m])$ in the first assumption
of Theorem~\ref{thm:sharp}
is directly related to the mean field free energy landscape.
In particular, it implies uniqueness of the invariant
measure.
We show below that this bound can be verified for models exhibiting flat
or displacement convexity across the full parameter regime, since such
convexities rule out phase transitions. For models with a second-order
phase transition, heuristic arguments indicate that this bound
remains valid throughout the entire subcritical regime;
we establish this rigorously for a quartic Curie--Weiss model in the following.

When the force kernel $\nabla_1W$ is bounded,
the second assumption follows from Pinsker's inequality,
while the third holds trivially.
However, when $\nabla_1W$ is merely Lipschitz,
verifying the transport inequality for $\mu = m^{N,k|[k-1]}_{*,\x^{[k-1]}}$
becomes unobvious, especially due to the effects of conditioning.
It is worth noting that the transport inequality for the conditional measure
is also needed in Lacker's reversed entropy result
\cite[Theorem~2.8]{LackerQuantitative},
though there it is only shown to hold in strongly log-concave cases.
We address this difficulty
using a \emph{$\TTE1$ tightening} argument
and a \emph{defective $\TTE2$ inequality},
both of which we believe have broader applications.

We begin by formulating the $\TTE1$ tightening.

\begin{thm}
\label{thm:t1-tightening}
Let $\mu \in \mathcal P_1(\X)$
and let $\rho > 0$, $\delta \geqslant 0$.
Suppose that for all $\nu \in \mathcal P_1(\X)$,
\begin{equation}
\label{eq:defective-t1}
W_1^2(\nu, \mu) \leqslant \frac 2\rho
\bigl(H(\nu | \mu) + \delta\bigr).
\end{equation}
Then for all $\nu \in \mathcal P_1(\X)$,
\begin{equation}
\label{eq:t1}
W_1^2(\nu, \mu) \leqslant \frac 8\rho (2+\delta)^2 H(\nu|\mu).
\end{equation}
\end{thm}

\begin{rem}
The $\TTE1$ tightening is an easy consequence of the Bolley--Villani inequality
\cite{BolleyVillaniCKP}
and, interestingly, does not require
any additional information on small scales, i.e.,
when $H(\nu | \mu) \ll \delta$.
This contrasts with the case of log-Sobolev inequality,
which requires an additional Poincaré for achieving a quantitative tightening.%
\footnote{A defective log-Sobolev inequality implies the existence
of a spectral gap, i.e., Poincaré, but no quantitative constant is possible.
See Theorem~1 and Proposition~11
in the work of \textcite{MicloHyperboundedness}.}
See e.g.\ \cite[Proposition~5.1.3]{BGLMarkov} for the standard
log-Sobolev tightening proof using
the Rothaus lemma, or \cite[Proposition~5]{nulsi} for an alternative
proof based on hypercontractivity.
\end{rem}

Next, we derive a defective $\TTE2$ inequality for the $N$-particle system,
building on the non-linear $\TTE2$ inequality of the mean-field system, with
the defect representing an approximation error.
This finding extends the tensorization property
of the classical $\TTE2$ inequality,
as established by \textcite{MartonConcentration}
and \textcite{TalagrandTransportation}
in the absence of mean field interactions.

To state the result, we introduce the notion of positivity for kernel
functions. A symmetric kernel $U\colon\X\times\X\to\R$ is of \emph{positive
type} if, for any signed measure $\nu$ on $\X$ with $\int_{\X}\dd\nu=0$, it
satisfies $\iint_{\X^2} U(x,y) \nu^{\otimes 2}(\dd x\dd y) \geqslant 0$. For
$\X = \R^d$, we restrict $U$ to have quadratic growth and impose this condition
only for $\nu$ with finite second moments.
We also refer to this property as \emph{flat convexity},
as for positive kernels, the associated quadratic form
on the space of probability measures
is convex with respect to flat interpolations.

\begin{thm}
\label{thm:defective-t2}
Let $\X = \R^d$.
Suppose that $W = W_+ - W_-$ where both $W_+$ and $W_-$
are $\mathcal C^2$ and of positive type
with $\lVert\nabla_{1,2}^2 W_+ \rVert_{L^\infty} \leqslant L^+_W$
and $\lVert\nabla_{1,2}^2 W_- \rVert_{L^\infty} \leqslant L^-_W$.
Suppose that there exist
$\varepsilon \in [0,1/2]$ and $\lambda \geqslant 0$ such that for all
$\nu \in \mathcal P_2(\R^d)$,
\begin{equation}
\label{eq:mf-neg-t2}
\frac{\lambda}{2} W_2^2(\nu, m_*) \leqslant
- \frac {1+\varepsilon}2 \langle W_-, (\nu - m_*)^{\otimes 2}\rangle
+ H(\nu | m_*).
\end{equation}
Then for all $\nu^N \in \mathcal P_2(\R^{Nd})$,
\begin{equation}
\label{eq:ps-defective-t2}
\frac{\lambda_N}{2}
W_2^2(\nu^N, m_*^{\otimes N}) \leqslant
H(\nu^N|m^N_*) + \delta_N,
\end{equation}
where $\lambda_N$, $\delta_N$ are defined by
\begin{align}
\lambda_N &= \begin{cases}
\makebox[0pt][l]
{$\lambda - 3(1/\sqrt{2\varepsilon}+3)L^-_WN^{-1},$}
\hphantom{
\bigl( 3(1/\sqrt{2\varepsilon}+3) L^-_W + L^+_W\bigr) \Var m_*,
}
& \varepsilon \in (0,1/2], \\
\lambda - 2(\log N+3)L^-_WN^{-1},
& \varepsilon = 0,
\end{cases} \label{eq:def-lambda-N}\\
\delta_N &= \begin{cases}
\bigl( 3(1/\sqrt{2\varepsilon}+3) L^-_W + L^+_W\bigr) \Var m_*,
& \varepsilon \in (0,1/2], \\
\bigl( 2(\log N + 3)L^-_W + L^+_W\bigr) \Var m_*,
& \varepsilon = 0.
\end{cases} \label{eq:def-delta-N}
\end{align}
Here $\Var m_*$ denotes the sum of the variances of the individual
components for a random variable with distribution $m_*$.
Consequently, in the case $\lambda_N > 0$,
for all $\nu^N \in \mathcal P_1(\R^{Nd})$,
\begin{equation}
\label{eq:ps-t1}
W_1^2(\nu^N, m^N_*) \leqslant \frac{64(1+\delta_N)^2}{\lambda_N}
H(\nu^N| m^N_*).
\end{equation}
Moreover, the $\TTE1$ inequality \eqref{eq:ps-t1} with the same constant
holds for any marginal distribution of $m^N_*$.

In particular, if $W$ consists solely of a positive part, i.e., $W = W_+$,
and if $m_*$ satisfies a $\TTE2$ inequality
with constant $\lambda$, then \eqref{eq:ps-defective-t2} holds
with $\lambda_N = \lambda$ and $\delta_N = L_W^+ \Var m_*$.
Consequently, \eqref{eq:ps-t1} holds with the same constants.
\end{thm}

\begin{rem}
Ideally, one would prefer to replace the assumption \eqref{eq:mf-neg-t2}
with the weaker:
\[
\frac\lambda 2W_2^2(\nu, m_*)
\leqslant \frac 12 \langle W, (\nu - m_*)^{\otimes 2}\rangle
+ H(\nu | m_*)
\]
and establish \eqref{eq:ps-defective-t2} with
$\lambda_N = \lambda - O(1/N)$ and $\delta_N = O(1)$.
Such behavior is optimal as it is attained by Gaussian examples.
In other words, we would like to set $\varepsilon = 0$
and exploit the extra concentration provided by the convex energy part.
This does not appear feasible within our current proof method.
Nonetheless, in the special case where $W = - W_-$,
since $\varepsilon$ can be made arbitrarily small,
the theorem yields arbitrarily small loss in the concentration constant
$\lambda_N$, while $\delta_N$ remains $O(1)$.
\end{rem}

\begin{rem}
A similar theorem can be formulated for the case $\X = \T^d$
where the boundedness of $W_-$ replaces the constant $L_W^-$.
Notably, this approach allows us to eliminate
the dependency on $N$ in $\lambda_N$.
\end{rem}

As a side note, the key step of the proof of Theorem~\ref{thm:defective-t2}
readily yields a quantitative large deviation estimate
of the Jabin--Z.~Wang type.

\begin{cor}
\label{cor:jw}
In the general case stated in Theorem~\ref{thm:defective-t2},
if $\varepsilon \in (0,1/2]$ and $\lambda_N \geqslant 0$,
then the following holds:
\[
\log \int_{\R^{Nd}}
\exp \biggl( - \frac N2 \langle W, (\mu_{\vect x} - m_*)^{\otimes 2}\rangle
\biggr) m_*^{\otimes N}(\dd \vect x)
\leqslant 3\biggl(\frac{1}{\sqrt{2\varepsilon}} + 3\biggr) L_W^- \Var m_*.
\]
\end{cor}

\begin{rem}
Even when reformulated in the torus setting, this corollary does not recover
the original result of \textcite[Theorem~4]{JabinWang}, where a smallness
condition on the exponential moment of $W$ replaces the assumption
\eqref{eq:mf-neg-t2}. This limitation arises because our approach distinguishes
between the flat convex and concave components in the interaction energy $W$.
However, constructing a spectral decomposition that preserves
suitable controls for $W$ appears to be challenging.
\end{rem}

Theorems~\ref{thm:t1-tightening} and \ref{thm:defective-t2}
enable us to verify the conditions of Theorem~\ref{thm:sharp}
in a wide range of scenarios involving unbounded forces at infinity.
We first present a general result applicable to flat semi-convex
and displacement convex cases.
For concrete examples, including
the one-dimensional Coulomb gas and two-layer neural networks,
we refer readers to the discussion at the end of this section.
Notably, in all these cases,
Lacker's original approach \cite[Theorem~2.2]{LackerQuantitative}
yields only perturbative results for sufficiently high temperatures,
whereas our method overcomes this restriction.

\begin{prop}
\label{prop:exm}
Suppose that one of the following assumptions holds true.
\begin{itemize}
\item (Flat semi-convexity \dots)
The interaction kernel writes $W = W_+ - W_-$
with both $W_+$, $W_-$ being continuous and of positive type.
The probability measures $\Pi[m]$ satisfy a $\rho_0$-log-Sobolev
inequality uniformly in $m \in \mathcal P(\X)$ for some $\rho_0 > 0$.
Moreover, one of the followings holds.
\begin{itemize}
\item (\,\dots\ with bounded force)
The kernels $W$, $W_-$ satisfy
$\lVert \nabla_1W\rVert_{L^\infty} \leqslant M_W < \infty$
and $\lVert \nabla_1W_-\rVert_{L^\infty} \leqslant M_W^-
< \sqrt{\rho_0}/2$.
In this case, we set
\[
\rho = \rho_0\biggl(1 - \frac{2M^-_W}{\sqrt{\rho_0}}\biggr),\qquad
\gamma = 2M_W,\qquad
M = 4M_W^2.
\]
\item (\,\dots\ with Lipschitz force)
The state space $\X$ is $\R^d$.
The kernels $W_+$, $W_-$ are $\mathcal C^2$ and satisfy
$\lVert \nabla_{1,2}^2W_+\rVert_{L^\infty} \leqslant L_W^+ < \infty$
and $\lVert \nabla_{1,2}^2W_-\rVert_{L^\infty} \leqslant L_W^- < \rho_0/2$.
In this case, we set
\begin{align*}
\rho &= \rho_0\biggl( 1 - \frac{2L^-_W}{\rho_0}\biggr),\\
\gamma &= \frac{64(1+\delta_N)^2(L^+_W+L^-_W)^2}{\lambda_N},\\
M &= 4(L^+_W+L^-_W)^2
\biggl( \frac{\delta_N}{\lambda_NN} + \frac{d}{\rho_0}\biggr),
\end{align*}
where $\lambda_N$, $\delta_N$ are given by
\eqref{eq:def-lambda-N}, \eqref{eq:def-delta-N}
respectively with $\varepsilon$ replaced by $1/2$
and $\Var m_*$ by $d / \rho_0$.
\end{itemize}

\item (Displacement convexity)
The state space $\X$ is $\R^d$.
The confinement $V$ is $\mathcal C^2$ and satisfies
$\nabla^2V \succcurlyeq \kappa_V$ for $\kappa_V > 0$
and the interaction $W$ is given by
$W(x,y) = W'(x - y)$ for some $\mathcal C^2$ and even $W' \colon \R^d \to \R$
satisfying $\nabla^2W' \succcurlyeq 0$
and $\lVert \nabla^2W'\rVert_{L^\infty} \leqslant L_W$.
In this case, we set
\[
\rho = \frac{\kappa_V}{2(1 + L_W^2/\kappa_V^2)},\qquad
\gamma = \frac{2L_W^2}{\kappa_V},\qquad
M = 4L_W^2 \biggl(1 + \frac{L_W^2}{\kappa_V^2N}\biggr)\frac{d}{\kappa_V}.
\]
\end{itemize}
Then Theorem~\ref{thm:sharp} applies with the constants
$\rho$, $\gamma$, $M$ whenever
$N$ is large enough such that $\gamma \geqslant 0$.
\end{prop}

\begin{rem}
When $m_*$ satisfies a log-Sobolev inequality,
the uniform log-Sobolev inequality for $\Pi[m]$
can often be established using perturbation lemmas.
For instance, when $W$ is bounded, this condition can be verified
via the Holley--Stroock method~\cite{HolleyStroockLSI}.
In the case of Lipschitz continuous forces, the result can be obtained
by applying the approach of Aida--Shigekawa~\cite{AidaShigekawaLSI}.
For a detailed discussion of this condition, we refer readers
to our previous work~\cite{ulpoc}.
\end{rem}

To demonstrate the full power of our theorems,
we show sharp chaos estimates for the well-known quartic Curie--Weiss model
of ferromagnetism throughout its sub-critical regime.
We select this model as a simple yet
representative example of systems exhibiting a second-order phase transition,
and this choice should not be read as a fundamental limitation of our results.
We expect that extending Theorem~\ref{thm:sharp}
to more general settings will yield analogous sharp chaos estimates for
mean field $\mathrm{O}(n)$ models,
as discussed in \textcite[Section~2, (iv)]{BauerschmidtBodineauSimple}.

\begin{prop}
\label{prop:curie-weiss}
Suppose that $\X = \R$ and the confinement potential is given by
\[
V(x) = \frac{\theta}{4} x^4 + \frac{\sigma}{2} x^2
\]
for $\theta > 0$ and $\sigma \in \R$.
Suppose that the interaction is given by
\[
W(x,y) = - J x y
\]
for $J \in (0, \Jc)$, where
\[
\Jc \coloneqq
\frac{\int_{\R} e^{-V(x)} \dd x}
{\int_{\R} x^2e^{-V(x)} \dd x}.
\]
Then Theorem~\ref{thm:sharp} applies with
$m_*$ being the centered measure
$\propto e^{-V(x)} \dd x$ and
\begin{align*}
\rho &= \biggl(1 - \frac{J}{\Jc}\biggr)^{\!2}\rho_0,\\
\gamma &= \frac{64(1+\delta_N)^2J^2}{\lambda_N},\\
M &= 4J^2 \biggl( \frac{\delta_N}{\lambda_NN} + \frac{d}{\rho_0}\biggr),
\end{align*}
where $\rho_0$, $\lambda_N$, $\delta_N$ are given by
\begin{align*}
\rho_0 &=
\begin{cases}
\sigma, & \sigma \geqslant 1, \\
\exp \bigl( - 7(1-\sigma)^2\!/36\theta \bigr),
& \sigma < 1,
\end{cases} \\
\lambda_N &= \biggl( 1 - \frac{J}{\Jc}\biggr)\frac{\rho_0}{2}
- \frac{3\bigl(1/\sqrt{\min(\Jc/J-1,1)}+3\bigr)J}{N}, \\
\delta_N &=
3\bigl(1/\sqrt{\min(\Jc/J-1,1)}+3\bigr)\frac{dJ}{\rho_0},
\end{align*}
whenever $N$ is large enough such that $\lambda_N > 0$.
\end{prop}

\begin{rem}
If $J=0$, the Curie--Weiss model is free of interaction,
and if $J<0$, it falls into the flat convex case with Lipschitz force
of Proposition~\ref{prop:exm}.
\end{rem}

\begin{rem}
\label{rem:curie-weiss-critical}
A quick computation for the Curie--Weiss model shows that, if
\[
N > \frac{12\bigl(1/\sqrt{\min(\Jc/J-1,1)}+3\bigr)J}
{(1-J/\Jc)\rho_0},
\]
then $\lambda_N \geqslant (1-J/\Jc)\rho_0/4$, and therefore for all $k \in [N]$,
\[
\frac{\rho_0}{2}W_2^2(m^{N,k}_*,m_*^{\otimes k})
\leqslant H(m^{N,k}_* | m_*^{\otimes k}) \leqslant C_J \frac{k^2}{N^2}
\]
for some $C_J > 0$ satisfying
\[
\liminf_{J \to \Jc-} \frac{\log C_J}{\log(1 - J/\Jc)} \geqslant -18.
\]
This bound on the critical exponent is highly sub-optimal.
For example, by not doing the second part of the last step
in the proof of Theorem~\ref{thm:sharp},
we immediately get bounds with exponent $-14$ for $k = O(1)$.
Moreover, by applying Theorem~\ref{thm:defective-t2}
with $\varepsilon = 0$ instead of $\varepsilon > 0$ in the third step
of the proof, we can further reduce the exponent to $-10$
though this introduces polylogarithmic factors in $N$ as a trade-off.
We do not elaborate on these arguments in detail,
as they are unlikely to yield the optimal exponent.

Meanwhile, Bauerschmidt, Bodineau and Dagallier establish
in a parallel study~\cite{BBDCriterionFreeEnergy}
that the log-Sobolev constants for the $N$-particle Gibbs measures
vanish in the order of $1 - J/\Jc$ when $J \to \Jc-$,
which is the expected behavior for Curie--Weiss models.
This leads to the global reversed entropy estimate
\[
H(m_*^{\otimes N} | m_*^N )
= O \biggl( \Bigl( 1 - \frac{J}{\Jc}\Bigr)^{\!-1}
I(m_*^{\otimes N} | m_*^N)\biggr)
= O \biggl( \Bigl( 1 - \frac{J}{\Jc}\Bigr)^{\!-1}\biggr),
\]
where we used the fact that $I(m_*^{\otimes N}|m_*^N) = O(1)$.
See the reversed entropy part in \textcite{LackerQuantitative} for details.
A local bound on the Wasserstein distance then follows
from its subadditivity combined with the $\TTE2$ inequality
for $m^N_*$:
\begin{multline*}
W_2^2(m^{N,k}_*, m_*^{\otimes k})
= O \biggl( \frac kN W_2^2(m^N_*, m_*^{\otimes N})\biggr)
= O \biggl( \Bigl( 1 - \frac{J}{\Jc}\Bigr)^{\!-1} \frac kN
H(m_*^{\otimes N} | m^N_*)\biggr) \\
= O \biggl( \Bigl( 1 - \frac{J}{\Jc}\Bigr)^{\!-2} \frac kN\biggr).
\end{multline*}

At this stage, interpolations between the exponents of
$1 - J/\Jc$ and $k/N$ can already be carried out.
However, the correct critical behaviors for the Curie--Weiss size of chaos
problem remain unclear to the authors
and it is even less certain
whether our method or Lacker's reversed entropy method can be improved
to prove the best exponents.
\end{rem}

\bigskip

The remainder of this section provides a detailed discussion of the flat
convex case and outlines potential future directions.

\subsection*{Flat convexities}

We now revisit the flat convex scenario discussed in Proposition~\ref{prop:exm}
and begin by providing two examples of flat convex potentials.

\begin{exm}[Two-layer neural network]
We consider the simplest case of a two-layer neural network, where both the
input and output are one-dimensional. The network consists of $N$ neurons,
the $i$th characterized by parameters $x^i = (a^i, b^i, w^i) \in \R^3$.
Its feature mapping, which defines the input-output relation, is expressed as:
\[
\R \ni z \mapsto \frac{1}{N} \sum_{i \in [N]}
a^i \sigma(w^i z + b^i) \in \R,
\]
where $\sigma \colon \R \to \R$ is a bounded non-linear activation function.
The $1/N$ factor ensures mean field scaling.

In regression tasks, a central problem in supervised machine learning,
the objective is to determine parameters $\x^{[N]} = (x^1, \ldots, x^N)$
that minimize the mean $L^2$ loss function
for a random input-output pair $(Z, Y)$:
\[
\Expect\biggl[\Bigl\lvert
Y - \frac{1}{N} \sum_{i \in [N]}
a^i \sigma(w^i Z + b^i)
\Bigr\rvert^2\biggr].
\]
Expanding the square reveals an equivalent form:
\[
\frac{1}{N^2} \sum_{i, j \in [N]}
\Expect\bigl[ \bigl(Y - a^i \sigma(w^i Z + b^i)\bigr)
\cdot \bigl(Y - a^j \sigma(w^j Z + b^j)\bigr)\bigr],
\]
which brings us to the two-body potential
\[
W(x, x') \coloneqq
\Expect\bigl[ \bigl(Y - a \sigma(w Z + b)\bigr)
\cdot \bigl(Y - a' \sigma(w' Z + b')\bigr)\bigr],
\]
where $x' = (a', w', b')$.
The potential $W$ satisfies the positivity property,
as any signed measure $\mu$ with $\int_{\R^3} \dd \mu = 0$ verifies
\[
\int_{(\R^3)^2} W(x, x') \mu^{\otimes 2}(\dd x \dd x')
= \Expect \biggl[ \Bigl( \int_{\R^3}
\bigl(Y - a \sigma(w Z + b) \bigr) \mu(\dd x) \Bigr)^{\!2} \biggr] \geqslant 0.
\]
To verify other conditions of Proposition~\ref{prop:exm},
it is generally necessary to assume that the activation function $\varphi$
is bounded and smooth,
and that the data distribution $\Law(Z,Y)$ admits suitable moment controls.
We refer readers to Section~3 of our previous work \cite{ulpoc} for details.

The study of the $N$-particle Gibbs measure $m^N_*$ is motivated
by its role as the long-time limiting distribution of the $N$ neurons
under the noised gradient descent algorithm.
In this context, the quantitative mean field limit towards $m_*$
provides an upper bound on the loss of the final training outcome
of a network consisting of $N$ neurons.
We observe that our main result can be applied
to achieve various improvements over recent works,
such as \cite[Theorem~3]{KZCELSampling}
and \cite[Proposition~4]{MonmarcheSchuhNonAsymptotic}.
For further details on the setup and numerical experiments,
we direct readers to \cite[Section~3]{ulpoc}
and related studies \cite{HRSS,NWSConvexMFL,ChizatMFL}.
\end{exm}

\begin{exm}[One-dimensional Coulomb gas]
Let $\X = \R$.
Consider the following interaction potential:
\[
W(x,y) = - \lvert x - y\rvert.
\]
It induces the force kernel
\[
\partial_1 W(x,y) = - \sgn (x - y).
\]
Notably, the force between two particles is repulsive
and does not depend on the distance between them.
The kernel is called the one-dimensional Coulomb potential
as it is the fundamental solution to the Laplace equation
\[
\partial_1^2 W(x,y) = - 2 \delta_y.
\]
This property implies the positivity of the Fourier transform of $W$
and thus, by a Bochner-type argument, $W$ is a flat convex kernel.

One technical issue remains:
the potential does not fit within the $\mathcal C^1$ framework
of this paper due to the discontinuity of $\nabla_1W$.
To address this, we propose the following regularization:
\[
W_\varepsilon(x,y) = - (\lvert \cdot\rvert\star\rho_{2\varepsilon}) (x-y),
\]
where $\rho_{2\varepsilon}$ is the centered Gaussian distribution
of variance $2\varepsilon$.
This regularization preserves the positivity,
as for all signed measure $\nu$ with zero net mass and finite second moments,
\[
\langle W_\varepsilon, \nu^{\otimes 2}\rangle
= \langle W, (\nu \star \rho_\varepsilon)^{\otimes 2}\rangle \geqslant 0.
\]

Finally, we provide some remarks on the Gibbs measure $m_*$.
At low temperatures,
corresponding to replacing $V$, $W$ by $\beta V$, $\beta W$
with $\beta$ large, $m_*$ is expected to converge to the
Frostman equilibrium, which minimizes the energy functional
\[
\nu \mapsto
\langle V, \nu \rangle + \frac{1}{2} \langle W, \nu^{\otimes 2} \rangle.
\]
The analytic characterization of this equilibrium
implies in the case of quadratic confinement
that it is a uniform distribution on an interval.
For further details, we refer to the lecture notes of Serfaty
\cite{SerfatyLecturesCoulombRiesz}.
\end{exm}

Having introduced the two examples of flat convexity, we now compare the
results of this paper with those obtained through global approaches. Assume
we are in the flat convexity scenario with Lipschitz force, meaning $W = W^+$
and consequently $L_W^- = 0$. Applying Theorem~\ref{thm:sharp} and letting
$\alpha = L^+_W / \rho_0$, we obtain
\[
H(m^{N,k}_* | m_*^{\otimes k})
\leqslant 72 \alpha^2\bigl(1 + 64\alpha^2(1+d\alpha)^2\bigr)^3
\biggl(1+\frac{\alpha}{N}\biggr)
\frac{k^2 + \bigl(384 \alpha^2(1+d\alpha)^2 + 1\bigr) k}{N^2}.
\]
For the global approach, we rely on a result from \textcite{KZCELSampling},
which refines some of the arguments
from our earlier collaboration with Chen \cite{ulpoc}. Theorem~5 in
\cite{KZCELSampling} implies that for $N \geqslant 80\alpha$,\footnote{
This theorem relies on \cite[Assumption~5]{KZCELSampling},
which imposes an additional $L^\infty$ bound on $\nabla_1^2 W$. However, upon
examining the proof, we observe that this bound is not used quantitatively.}
\[
H(m^{N,k}_* | m_*^{\otimes k})
\leqslant \frac{33\alpha d}{2}\frac{k}{N}.
\]
We also note a recent work by \textcite{NitandaImproved}, which focuses
on two-layer neural networks and establishes a global entropy bound without
relying on the log-Sobolev constant. Nitanda’s setting restricts the force
kernel to be bounded, providing an interesting comparison to the results in
the bounded force case of Proposition~\ref{prop:exm}.
However, we do not compare Nitanda's results in detail
as their formulation is highly model-specific,
whereas our aim is to provide a more general framework.

For Lipschitz forces, our approach yields better asymptotics
for $k = o(N)$ but introduces a 19th-power dependency
on the dimensionless parameter $\alpha$.
As noted in Remark~\ref{rem:curie-weiss-critical},
certain improvements are achievable by selectively applying the arguments
in the proof of Theorem~\ref{thm:sharp}.
However, even with these refinements, the linear
dependency on $\alpha$ present in the global result remains unattainable.
We believe that the strong dependency on $\alpha$ is a proof artifact, though
the optimal dependency in the sharp chaos case remains uncertain. As the main
goal of this paper is to provide a unified perspective for various gradient
interactions, we leave further optimization of this dependency to future work.

\subsection*{Perspectives}

An immediate question arises: can the comparison between conditional measures
be extended to establish time-uniform sharp chaos in the dynamical setting
by leveraging the non-linear log-Sobolev inequality? Another avenue of interest
is removing the sub-optimal regime identified by \textcite{LLFSharp}
for gradient dynamics. Our response is partially affirmative.
The conditional comparison yields at least time-uniform bounds
for regular H-stable potentials or divergence-free forces on the torus.
The full argument will be developed in a future work.
For more general cases, however, the answer remains unclear. A significant
challenge lies in identifying an appropriate free energy functional that aligns
with the conditional gradient descent. Additionally, establishing
time-uniform concentration for dynamical conditional laws is
considerably more complex than in the static setting.

Another open question is the derivation of sharp chaos bounds for equilibrium
measures of Coulomb particles in $d \geqslant 2$. Specifically, this involves
developing a localized version of the results by
\textcite{RougerieSerfatyHigher}. Intuitively, this requires identifying a local
variant of the modulated free energy, as entropy no longer adequately controls
the Coulomb energy, except in the high-temperature regime for $d = 2$.
Furthermore, our approach heavily relies on the dissipation of free energy.
For the Coulomb interaction, this mechanism was only recently understood
by \textcite{RosenzweigSerfatyMLSI}
in the one-dimensional case via a modulated log-Sobolev inequality.
Extending this knowledge to higher dimensions, however, remains challenging.

It goes without saying that the dynamical sharp chaos problem for singular
kernels combines the difficulties inherent in both of these aspects.

\section{Proof of Theorem~\ref{thm:sharp}}
\label{sec:proof-sharp}

The proof focuses on finding inequalities between conditional entropies
\[
H_k \coloneqq
\int_{\X^{k-1}}
H(m^{N,k|[k-1]}_{*, \x^{[k-1]}}| m_*)
m^{N,k-1}_*(\dd\x^{[k-1]}),\qquad k \in [N]
\]
and is split into five steps.
Specifically, in the first three steps we derive a bound
of $H_k$ in terms of $H_{k+1}$ for $k < N$.
The fourth step handles the boundary term $H_N$,
and the final step solves the resulting system of inequalities.

We will employ arguments similar to those in
\cite[Proof of Theorem~2.2]{LackerQuantitative}
to control the error terms arising from the particle approximation.
The key novelty of our method, as highlighted in the introduction,
lies in comparing neighboring conditional measures in Step~1,
which subsequently leads to higher-order entropy differences in Step~3.

\proofstep{Step 1: Inserting a conditional measure}
The explicit density of the Gibbs measure
\[
m^{N}_* \propto \exp \biggl( - \sum_{i \in [N]} V(x^i)
- \frac 1{2N} \sum_{i,j \in [N] } W(x^i, x^j) \biggr)
\]
implies that its marginal distribution $m^{N,k}_*$ satisfies
\begin{equation}
\label{eq:log-density}
\nabla_k \log m^{N,k}_*(\x^{[k]})
+ \nabla V(x^k)
+ \biggl\langle \nabla_1 W(x^k,\cdot),
\frac kN \mu_{\x^{[k]}}
+ \frac {N-k}{N} m^{N,k+1|[k]}_{*,\x^{[k]}} \biggr\rangle = 0.
\end{equation}
See \cite[Lemma~3.1]{LackerQuantitative} for details.
As observed in our previous work \cite[Proof of Theorem~2.6]{ulpoc},
the log-density $\log m^{N,k}_*$ admits, by definition, the decomposition
\[
\log m^{N,k}_*(\x^{[k]})
= \log m^{N,k|[k-1]}_{*,\x^{[k-1]}}(x^k)
+ \log m^{N,k-1}_*(\x^{[k-1]}).
\]
As the last term does not depend on $x^k$, we have
\[
\nabla_k\log m^{N,k}_*(\x^{[k]})
= \nabla_k \log m^{N,k|[k-1]}_{*,\x^{[k-1]}}(x^k).
\]
Now focus on the case $k < N$.
Adding $\langle \nabla_1 W(x^k,\cdot), m^{N,k|[k-1]}_{*,\x^{[k-1]}}\rangle$
to both sides of \eqref{eq:log-density} yields
\begin{equation}
\label{eq:insert-conditional}
\begin{IEEEeqnarraybox}[][c]{rCl}
\IEEEeqnarraymulticol{3}{l}{
\nabla_k \log m^{N,k|[k-1]}_{*, \x^{[k-1]}}(x^{k})
+ \nabla V(x^{k})
+ \langle \nabla_1 W(x^k,\cdot), m^{N,k|[k-1]}_{*,\x^{[k-1]}}\rangle} \\
\quad&=&
- \frac kN
\langle \nabla_1 W(x^k,\cdot), \mu_{\x^{[k]}}-m_*
\rangle \\
&&\adjustbin
+ \frac kN\langle \nabla_1 W(x^k,\cdot), m^{N,k+1|[k]}_{*,\x^{[k]}} - m_*
\rangle \\
&&\adjustbin -
\langle \nabla_1 W(x^k,\cdot), m^{N,k+1|[k]}_{*,\x^{[k]}}
- m^{N,k|[k-1]}_{*,\x^{[k-1]}}\rangle
\end{IEEEeqnarraybox}
\end{equation}
The left hand side is nothing but
\[
\nabla_k \log
\frac{m^{N,k|[k-1]}_{*, \x^{[k-1]}}(x^{k})}
{\Pi[m^{N,k|[k-1]}_{*, \x^{[k-1]}}](x^{k})},
\]
as the definition \eqref{eq:def-Pi} of the mapping $\Pi$ gives
\[
\nabla_k \log \Pi[m^{N,k|[k-1]}_{*, \x^{[k-1]}}](x^{k})
= - \nabla V(x^k) - \langle \nabla_1 W(x^k,\cdot),
m^{N,k|[k-1]}_{*, \x^{[k-1]}} \rangle.
\]
We take the square of the equality \eqref{eq:insert-conditional}
and integrate it against the probability measure $m^{N,k}_*$
to obtain
\begin{IEEEeqnarray*}{rCl}
\IEEEeqnarraymulticol{3}{l}
{\int_{\X^{k-1}}
I(m^{N,k|[k-1]}_{*, \x^{[k-1]}}| \Pi[m^{N,k|[k-1]}_{*, \x^{[k-1]}}])
m^{N,k-1}_*(\dd\x^{[k-1]})} \\
\quad&=&
\int_{\X^k}
\biggl\lvert
\nabla_k \log
\frac{m^{N,k|[k-1]}_{*, \x^{[k-1]}}(x^{k})}
{\Pi[m^{N,k|[k-1]}_{*, \x^{[k-1]}}](x^{k})}
\biggr\rvert^2 m^{N,k}_*(\dd\x^{[k]}) \\
&\leqslant&
\frac{3k^2}{N^2}
\int_{\X^k} \lvert \langle \nabla_1 W(x^k, \cdot),
\mu_{\x^{[k]}} - m_*\rangle\rvert^2 m^{N,k}_*(\dd\x^{[k]}) \\
&& \adjustbin + \frac{3k^2}{N^2}
\int_{\X^k} \lvert \langle \nabla_1 W(x^k, \cdot),
m^{N,k+1|[k]}_{*,\x^{[k]}}
- m_*\rangle\rvert^2 m^{N,k}_*(\dd\x^{[k]}) \\
&& \adjustbin + 3
\int_{\X^k} \lvert \langle \nabla_1 W(x^k, \cdot),
m^{N,k+1|[k]}_{*, \x^{[k]}}
-m^{N,k|[k-1]}_{*, \x^{[k-1]}}
\rangle\rvert^2 m^{N,k}_*(\dd\x^{[k]}) \\
&\eqqcolon& T_1+T_2+T_3,
\end{IEEEeqnarray*}
the inequality being Cauchy--Schwarz.
The non-linear log-Sobolev inequality
applied to the measure $m^{N,k|[k-1]}_{*,\x^{[k-1]}}$ gives
\[
2\rho H(m^{N,k|[k-1]}_{*, \x^{[k-1]}}| m_*)
\leqslant I(m^{N,k|[k-1]}_{*, \x^{[k-1]}}| \Pi[m^{N,k|[k-1]}_{*, \x^{[k-1]}}]).
\]
Therefore we have shown
\[
2 \rho H_k \leqslant T_1 + T_2 + T_3.
\]
We aim to find upper bounds for $T_1$, $T_2$ and $T_3$
in the following two steps.

\proofstep{Step 2: Bounding $T_1$}
By the Cauchy--Schwarz inequality,
\begin{multline*}
\int_{\X^k} \lvert \langle \nabla_1 W(x^k, \cdot),
\mu_{\x^{[k]}} - m_*\rangle\rvert^2 m^{N,k}_*(\dd\x^{[k]}) \\
\leqslant
\frac 1k\sum_{i\in [k]}\int_{\X^k}
\lvert \nabla_1 W(x^k, x^i) -
\langle\nabla_1 W(x^k,\cdot),m_*\rangle\rvert^2 m^{N,2}_*(\dd x^1\dd x^k)
\leqslant M.
\end{multline*}
The first crude bound follows:
\begin{equation}
\label{eq:t1-crude-bound}
T_1 \leqslant \frac{3Mk^2}{N^2}.
\end{equation}
For simplicity of notation,
we introduce the \emph{reduced interaction potential}
\begin{equation}
\label{eq:def-w-reduced}
W_*(x,y) \coloneqq W(x,y)
- \langle W(x,\cdot), m_*\rangle
- \langle W(\cdot,y), m_*\rangle
+ \langle W(\cdot,\cdot),m_*^{\otimes 2}\rangle.
\end{equation}
For the finer bound, observe that
\begin{IEEEeqnarray*}{rCl}
T_1 &=& \frac{3}{N^2}
\int_{\X^k} \Bigl\lvert
\sum_{i \in [k]}
\nabla_1 W_*(x^1,x^i)\Bigr\rvert^2
m^{N,k}_*(\dd\x^{[k]}) \\
&=& \frac{3(k-1)(k-2)}{N^2}
\int_{\X^3} \nabla_1 W_*(x^1,x^2) \cdot \nabla_1 W_*(x^1,x^3)
m^{N,3}_*(\dd x^1\dd x^2\dd x^3) \\
&&\adjustbin + \frac{3(k-1)}{N^2}
\int_{\X^2} \lvert\nabla_1 W_*(x^1,x^2)\rvert^2 m^{N,2}_*(\dd x^1\dd x^2) \\
&&\adjustbin + \frac{6(k-1)}{N^2}
\int_{\X^2}
\nabla_1 W_*(x^1,x^1)\cdot
\nabla_1 W_*(x^1,x^2) m^{N,2}_*(\dd x^1\dd x^2) \\
&&\adjustbin + \frac{3}{N^2}
\int_{\X}
\lvert\nabla_1 W_*(x^1,x^1)\rvert^2
m^{N,1}_*(\dd x^1)
\end{IEEEeqnarray*}
By Cauchy--Schwarz,
the sum of the last three terms can be bounded by
\[
\frac{3M(3k-2)}{N^2}.
\]
Conditioning on $x^1$ and $x^2$ in the first term yields
\begin{IEEEeqnarray*}{rCl}
\IEEEeqnarraymulticol{3}{l}
{\int_{\X^3} \nabla_1 W_*(x^1,x^2) \cdot \nabla_1 W_*(x^1,x^3)
m^{N,3}_*(\dd x^1\dd x^2\dd x^3)} \\
\quad&=& \int_{\X^2} \nabla_1 W_*(x^1,x^2)
\cdot \langle \nabla_1 W_*(x^1,\cdot), m^{N,3|[2]}_{*,\x^{[2]}}
\rangle m^{N,2}_* (\dd x^1\dd x^2) \\
&\leqslant&
\sqrt M
\biggl(
\int_{\X^2} \lvert
\langle \nabla_1 W_*(x^1), m^{N,3|[2]}_{*,\x^{[2]}}\rangle
\rvert^2 m^{N,2}_*(\dd x^1\dd x^2)
\biggr)^{\!1/2} \\
&\leqslant& \sqrt{\gamma M H_3},
\end{IEEEeqnarray*}
where the last line follows from the transport inequality.
We have thus obtained the finer bound:
\begin{equation}
\label{eq:t1-fine-bound}
T_1 \leqslant \frac{3}{N^2}
\bigl( (3k-2)M + (k-1)(k-2) \sqrt{\gamma MH_3} \bigr).
\end{equation}

\proofstep{Step 3: Bounding $T_2$ and $T_3$}
Applying the transport inequality for $m_*$ yields
\[
T_2 \leqslant \frac{3\gamma k^2}{N^2}
\int_{\X^k} H(m^{N,k+1|[k]}_{*,\x^{[k]}} | m_*) m^{N,k}_*(\dd\x^{[k]})
= \frac{3\gamma k^2}{N^2}H_{k+1}.
\]
Likewise, by the transport inequality for $m^{N,k|[k-1]}_{*,\x^{[k-1]}}$,
\begin{IEEEeqnarray*}{rCl}
T_3 &\leqslant& 3 \gamma
\int_{\X^k} H(m^{N,k+1|[k]}_{*,\x^{[k]}} |
m^{N,k|[k-1]}_{*,\x^{[k-1]}}) m^{N,k}_*(\dd\x^{[k]}) \\
&=&
3\gamma \int_{\X^k} \int_{\X}
\log \frac{m^{N,k+1|[k]}_{*,\x^{[k]}}(x^{k+1})}
{m^{N,k|[k-1]}_{*,\x^{[k-1]}}(x^{k+1})}
m^{N,k+1|[k]}_{*,\x^{[k]}}(\dd x^{k+1})
m^{N,k}_*(\dd\x^{[k]}) \\
&=& 3\gamma
\int_{\X^{k+1}}
\log \frac{m^{N,k+1|[k]}_{*,\x^{[k]}}(x^{k+1}) / m_*(x^{k+1})}
{m^{N,k|[k-1]}_{*,\x^{[k-1]}}(x^{k+1}) / m_*(x^{k+1})}
m^{N,k+1}_*(\dd\x^{[k+1]}) \\
&=& 3\gamma
\int_{\X^{k+1}}
\log \frac{m^{N,k+1|[k]}_{*,\x^{[k]}}(x^{k+1})}{m_*(x^{k+1})}
m^{N,k+1}_*(\dd\x^{[k]}) \\
&&\adjustbin- 3\gamma \int_{\X^{k}}
\log \frac{m^{N,k|[k-1]}_{*,\x^{[k-1]}}(x^{k+1})}{m_*(x^{k+1})}
m^{N,k}_*(\dd\x^{[k-1]}\dd x^{k+1})
\\
&=& 3\gamma (H_{k+1} - H_k),
\end{IEEEeqnarray*}
where $H_{k+1} - H_k$ is the higher-order entropy difference mentioned
in the beginning of the proof.
This chain of equalities implies in particular that
for $k < N$,
\[
H_{k+1} \geqslant H_k.
\]
Hence $T_2$ is further upper bounded by
$\frac{3\gamma k^2}{N^2}H_N$.

To summarize, we have established that for $k < N$,
\begin{equation}
\label{eq:hk}
2\rho H_k \leqslant T_1 + \frac{3\gamma k^2}{N^2} H_N
+ 3\gamma (H_{k+1}-H_k),
\end{equation}
where $T_1$ admits the crude upper bound \eqref{eq:t1-crude-bound}
and the refined \eqref{eq:t1-fine-bound}.
To complete the system of inequalities,
we study the boundary term $H_N$ in the next step.

\proofstep{Step 4: Bounding the boundary term $H_N$}
Take $k = N$ in \eqref{eq:log-density}
and compare it with the mean field stationary condition
\[
\nabla \log m_*(x)
+ \nabla V(x) + \langle \nabla_1W(x,\cdot), m_*\rangle = 0.
\]
This gives the following equality:
\[
\nabla_N \log \frac{m^{N,N|[N-1]}_{*,\x^{[N-1]}}(x^N)}
{m_*(x^N)}
+ \langle \nabla_1 W(x^N, \cdot), \mu_{\x^{[N]}} - m_*\rangle = 0.
\]
By the log-Sobolev inequality for $m_*$,
\begin{IEEEeqnarray*}{rCl}
2\rho H_N
&=& 2\rho \int_{\X^{N-1}}
H(m^{N,N|[N-1]}_{*,\x^{[N-1]}} | m_*) m^{N,N-1}_*(\dd\x^{[N-1]}) \\
&\leqslant&
\int_{\X^{N-1}} I(m^{N,N|[N-1]}_{*,\x^{[N-1]}} | m_*)
m^{N,N-1}_*(\dd\x^{[N-1]}) \\
&=&
\int_{\X^{N}}
\lvert\langle \nabla_1 W_*(x^N, \cdot), \mu_{\x^{[N]}}\rangle\rvert^2
m^{N}_*(\dd\x^{[N]}).
\end{IEEEeqnarray*}
The Cauchy--Schwarz inequality gives that
\begin{multline*}
\int_{\X^{N}}
\lvert\langle \nabla_1 W_*(x^N, \cdot), \mu_{\x^{[N]}}\rangle\rvert^2
m^{N}_*(\dd\x^{[N]}) \\
\leqslant
\frac 1N \sum_{i \in [N]}
\int_{\X^N}
\lvert\langle \nabla_1 W_*(x^N, x^i) \rangle\rvert^2
m^{N}_*(\dd\x^{[N]})
\leqslant M.
\end{multline*}
Hence the crude bound follows:
\begin{equation}
\label{eq:hn-crude-bound}
2\rho H_N \leqslant M.
\end{equation}
Alternatively, using the method for the finer bound \eqref{eq:t1-fine-bound}
in Step~2 gives
\begin{equation}
\label{eq:hn-fine-bound}
2\rho H_N \leqslant
\frac{1}{N^2} \bigl((3N-2) M + (N-1)(N-2) \sqrt{\gamma MH_3}\bigr).
\end{equation}

\proofstep{Step 5: Solving the system of inequalities}
Having established the upper bounds for $T_1$, $T_2$ and $T_3$
and the boundary $H_N$,
we are now ready to solve the system of inequalities
and prove the desired upper bounds for $H_k$.
This will be done by following two-step argument.
We first apply the crude bounds for $T_1$ and $H_N$
and solve the system of inequalities
to get the suboptimal $H_k = O(k^2\!/N^2)$.
Then we apply the finer bounds,
using the information obtained from the crude ones,
and show the sharp control $H_k = O(k/N^2)$.

Inserting the crude bounds
\eqref{eq:t1-crude-bound} and \eqref{eq:hn-crude-bound}
for $T_1$ and $H_N$ into \eqref{eq:hk} leads to
\[
2\rho H_k \leqslant \frac{3M\bigl(1+\frac\gamma{2\rho}\bigr)k^2}{N^2}
+ 3 \gamma (H_{k+1} - H_k)
\]
for $k < N$, the boundary satisfying $H_N \leqslant M/2\rho$.
Let us define
\[
\bar H_k = \frac{3M\bigl( 1 + \frac{\gamma}{2\rho}\bigr)}{\rho N^2}
\biggl(k^2+\frac{6\gamma}{\rho}k+\frac{18\gamma^2}{\rho^2}
+\frac{3\gamma}{\rho}\biggr).
\]
The construction ensures that
\[
2\rho \bar H_k \geqslant \frac{3M\bigl(1+\frac\gamma{2\rho}\bigr)k^2}{N^2}
+ 3 \gamma (\bar H_{k+1} - \bar H_k),
\]
as we can verify
\begin{align*}
\rho \bar H_k &\geqslant
\frac{3M \bigl(1+\frac\gamma{2\rho}\bigr)}{N^2}k^2,\\
\rho \bar H_k &\geqslant
\frac{9M\gamma\bigl( 1 + \frac\gamma{2\rho}\bigr)}{\rho N^2}
\biggl( 2k + \frac{6\gamma}{\rho} + 1\biggr)
= 3 \gamma (\bar H_{k+1} - \bar H_k).
\end{align*}
In the language of analysis,
$\bar H_k$ is an upper solution to the difference equation.
Moreover, on the boundary, $\bar H_N \geqslant H_N$.
Considering the system of inequalities satisfied
by the difference $H_k - \bar H_k$,
we see that $H_k \leqslant \bar H_k$ for all $k \in [N]$, namely,
\[
H_k \leqslant
\frac{3M\bigl( 1 + \frac{\gamma}{2\rho}\bigr)}{\rho N^2}
\biggl(k^2+\frac{6\gamma}{\rho}k+18\Bigl(\frac{\gamma}{\rho}\Bigr)^{\!2}
+3\frac\gamma\rho\biggr).
\]
In particular, taking $k=3$ gives
\[
H_3 \leqslant \frac{9M\bigl(1+\frac{\gamma}{\rho}\bigr)
\bigl(3+\frac{7\gamma}{\rho}+\frac{6\gamma^2}{\rho^2}\bigr)}
{\rho N^2}
\leqslant \frac{54M\bigl(1+\frac\gamma\rho\bigr)^3}{\rho N^2}.
\]

Now we insert the control on $H_3$ into the finer bounds
\eqref{eq:t1-fine-bound} and \eqref{eq:hn-fine-bound} to find
\begin{align*}
T_1 &\leqslant \frac{9Mk}{N^2} + \frac{9\sqrt{6}
\sqrt{\frac\gamma\rho\bigl(1+\frac\gamma\rho\bigr)^3}Mk^2}{N^3}
\leqslant \frac{36\bigl(1+\frac\gamma\rho\bigr)^2Mk}{N^2}, \\
2\rho H_N &\leqslant \frac{3M}{N} + \frac{3\sqrt{6}
\sqrt{\frac\gamma\rho\bigl(1+\frac\gamma\rho\bigr)^3}M}{N}
\leqslant \frac{12\bigl(1+\frac\gamma\rho\bigr)^2M}{N},
\end{align*}
the second inequalities on both lines being a consequence of
the following:
\[
\forall t \geqslant 0,\qquad
\frac{3 + 3\sqrt 6 \sqrt{t(1+t)^3}}{(1+t)^2}
\leqslant 3 + 3\sqrt 6 < 12.
\]
Applying the upper bound on $T_1$ to \eqref{eq:hk},
we obtain the system of inequalities
\begin{align*}
2\rho H_k
&\leqslant \frac{36\bigl(1+\frac\gamma\rho\bigr)^2Mk}{N^2}
+ \frac{3\gamma k^2}{N^2}
\cdot \frac{6\bigl(1+\frac\gamma\rho\bigr)^2M}{\rho N}
+ 3\gamma (H_{k+1} - H_k) \\
&\leqslant \frac{36\bigl(1+\frac\gamma\rho\bigr)^3Mk}{N^2}
+ 3\gamma (H_{k+1} - H_k).
\end{align*}
Let us define
\[
\bar H'_k = \frac{36\bigl(1+\frac\gamma\rho\bigr)^3M}{\rho N^2}
\biggl( k + \frac{3\gamma}{\rho}\biggr).
\]
We again verify that $\bar H'_k$ is an upper solution
to the system of inequalities above with $\bar H'_N \geqslant H_N$.
By the same comparison argument, we deduce that for all $k \in [N]$,
\[
H_k \leqslant \frac{36\bigl(1+\frac\gamma\rho\bigr)^3M}{\rho N^2}
\biggl( k + \frac{3\gamma}{\rho}\biggr),
\]
which is the first claim of the theorem.
The second claim follows by the chain rule:
\[
H(m^{N,k}_*|m_*^{\otimes k})
= H_1 + \cdots + H_k
\leqslant \frac{18\bigl(1+\frac\gamma\rho\bigr)^3M}{\rho N^2}
\biggl( k^2 + \Bigl(1+\frac{6\gamma}{\rho}\Bigr)k\biggr).
\tag*{\qed}
\]

\section{Proof of Theorem~\ref{thm:t1-tightening}}

Let $\varphi\colon\X \to \R$ be a $1$-Lipschitz function
such that $\langle\varphi,\mu\rangle = 0$.
By the argument of \textcite{BobkovGoetzeExponential},
the defective $\TTE1$ inequality \eqref{eq:defective-t1} implies that
for all $\lambda \in\R$,
\[
\log \langle e^{\lambda \varphi}, \mu\rangle
\leqslant \frac{\lambda^2}{2\rho} + \delta.
\]
Choosing $\varphi$ as mean-zero coordinate functions
shows that the probability measure $\mu$ has a Gaussian tail,
and hence satisfies a $\TTE1$ inequality
by \textcite{DGWTransportation}.
To quantify the $\TTE1$ constant, we apply the Bolley--Villani inequality
\cite[Theorem~2.1(ii)]{BolleyVillaniCKP}
to the function $\sqrt{\rho}\,\varphi / 2$.
This yields for all $\nu \in \mathcal P_1(\X)$,
\[
\frac{\rho}{4}\lvert\langle \varphi, \nu - \mu\rangle\rvert^2
\leqslant 2 \biggl( \frac 32 + \log \int_{\X}
e^{\rho\varphi^2\!/4} \dd\mu \biggr)^{\!2} H(\nu|\mu).
\]
It remains to control the Gaussian moment of $\varphi$,
which we achieve by the trick in \cite[p.~2704]{DGWTransportation}.
Let $\mathcal N$ be an standard Gaussian variable in $\R$.
We compute:
\begin{align*}
\int_{\X} e^{\rho\varphi^2\!/4} \dd\mu
= \int_{\X} \Expect [ e^{\sqrt{\rho/2}\,\varphi\mathcal N} ] \dd \mu
= \Expect \biggl[ \int_{\X} e^{\sqrt{\rho/2}\,\varphi\mathcal N} \dd \mu \biggr]
\leqslant \Expect[
e^{\mathcal N^2\!/4 + \delta}]
= \sqrt{2}\,e^{\delta}.
\end{align*}
Substituting this estimate into the Bolley--Villani gives
\[
\frac{\rho}{4}\lvert\langle \varphi, \nu - \mu\rangle\rvert^2
\leqslant 2 \biggl( \frac 32 + \frac 12\log 2 + \delta \biggr)^{\!2} H(\nu|\mu)
\leqslant 2 ( 2 + \delta )^{2} H(\nu|\mu).
\]
Taking the supremum over all $1$-Lipschitz $\varphi$
with $\langle\varphi,\mu\rangle=0$
completes the proof.\qed

\section{Proof of Theorem~\ref{thm:defective-t2} and Corollary~\ref{cor:jw}}

The proof is organized into five steps.
The initial four steps are dedicated
to establishing Theorem~\ref{thm:defective-t2},
with the second step providing a crucial lower bound
for a modulated free energy.
In the final step, we demonstrate that this key estimate directly leads to
Corollary~\ref{cor:jw}.

\proofstep{Step 1: Modulated free energy}
We recall from the work of Bresch, Jabin and Z.~Wang \cite{BJWAttractive}
the definition of \emph{modulated free energy}:
for $\nu^N \in \mathcal P_2(\mathbb R^d)$, we define
\begin{align*}
\mathcal F^N(\nu^N | m_*)
&\coloneqq \frac N2\Expect_{\bm{X} \sim \nu^N}
[\langle W, (\mu_{\bm{X}} - m_*)^{\otimes 2}\rangle]
+ H(\nu^N | m_*^{\otimes N}) \\
&= \frac N2\Expect_{\bm{X} \sim \nu^N}
[\langle W_*, \mu_{\bm{X}}^{\otimes 2}\rangle]
+ H(\nu^N | m_*^{\otimes N}),
\end{align*}
where $W_*$ denotes the reduced potential defined by \eqref{eq:def-w-reduced}.
Since the second argument in the modulated free energy
is the mean field Gibbs measure $m_*$,
by employing the fixed point relation \eqref{eq:m*-fixed-point},
we observe that it is precisely the difference between free energies:
\[
\mathcal F^N(\nu^N | m_*)
= \mathcal F^N(\nu^N) - N \mathcal F(m_*),
\]
where $\mathcal F^N$ are $\mathcal F$ are free energy functionals
defined at the beginning of the introduction.
Such free energy difference is the central quantity considered
in our previous work \cite{ulpoc}.
By definition of $m^N_*$, we have
\[
\mathcal F^N(\nu^N) - \mathcal F^N(m^N_*)
= H(\nu^N | m^N_*).
\]
Therefore, the relative entropy in the desired defective $\TTE2$ inequality
\eqref{eq:ps-defective-t2} can be expressed as
\begin{equation}
\label{eq:ps-relative-entropy-modulated-free-energy}
H(\nu^N | m^N_*)
= \mathcal F^N(\nu^N|m_*)
- \mathcal F^N(m^N_*|m_*).
\end{equation}
In the following two steps, we establish, respectively,
a lower bound and an upper bound for the two terms
in the right hand side of \eqref{eq:ps-relative-entropy-modulated-free-energy}.

\proofstep{Step 2: Lower bound of $\mathcal F^N(\nu^N|m_*)$}
This is the main step of the proof.
By the decomposition $W = W_+ - W_-$,
\begin{IEEEeqnarray*}{rCl}
\mathcal F^N(\nu^N) - N\mathcal F(m_*)
&=& \frac{N}{2} \Expect [ \langle W_*, \mu_{\bm{X}}^{\otimes 2}\rangle]
+ H(\nu^N | m_*^{\otimes N}) \\
&\geqslant&
- \frac{N}{2} \Expect [ \langle W_{-,*}, \mu_{\bm{X}}^{\otimes 2}\rangle]
+ H(\nu^N | m_*^{\otimes N}),
\end{IEEEeqnarray*}
where $W_{-,*}$ is the reduced potential defined by
\[
W_{-,*}(x,y)
\coloneqq W_-(x,y) - \langle W_-(\cdot, y), m_*\rangle
- \langle W_-(x,\cdot), m_*\rangle
+ \langle W_-(\cdot,\cdot), m_*^{\otimes 2}\rangle.
\]
Let $S_N$ be the permutation group on the $N$-element set.
The chain rule for relative entropy gives, for each $\sigma \in S_N$,
\[
H(\nu^N | m_*^{\otimes N})
= \sum_{k=1}^N
\Expect [ H( \nu^{N,\sigma(k) | \sigma([k-1])}_{\bm{X}^{\sigma([k-1])}} | m_*)],
\]
where $\bm{X}$ is a random variable distributed according to the law $\nu^N$
and $\nu^{N,\sigma(k) | \sigma([k-1])}_{\bm{X}^{\sigma([k-1])}}$
is the conditional measure defined by
\[
\nu^{N,\sigma(k) | \sigma([k-1])}_{\bm{X}^{\sigma([k-1])}}
= \Law(X^{\sigma(k)} | \bm{X}^{\sigma([k-1])})
= \Law(X^{\sigma(k)} | X^{\sigma(1)}, \ldots, X^{\sigma(k-1)}).
\]
Motivated by this chain rule structure,
we apply the non-linear $\TTE2$ inequality \eqref{eq:mf-neg-t2}
to each of the conditional measures
$\nu^{N,\sigma(k) | \sigma([k-1])}_{\bm{X}^{\sigma([k-1])}}$
and sum over $k \in [N]$, to obtain
\begin{multline*}
- \frac {(1+\varepsilon)}2 \sum_{k \in [N]}
\Expect [ \langle W_{-,*},
(\nu^{N,\sigma(k) | \sigma([k-1])}_{\bm{X}^{\sigma([k-1])}})^{\otimes 2}\rangle
] + H(\nu^N | m_*^{\otimes N}) \\
\geqslant \frac{\lambda}{2}
\sum_{k \in [N]}
\Expect[W_2^2(\nu^{N,\sigma(k) | \sigma([k-1])}_{\bm{X}^{\sigma([k-1])}}, m_*)].
\end{multline*}
The right hand side provides a transport plan
between $\nu^N$ and $m_*^{\otimes N}$,
and therefore dominates the $W_2$ distance:
\[
\sum_{k \in [N]}
\Expect[W_2^2(\nu^{N,\sigma(k) | \sigma([k-1])}_{\bm{X}^{\sigma([k-1])}}, m_*)]
\geqslant W_2^2(\nu^N, m_*^{\otimes N}).
\]
We refer readers to \cite[(22.9)]{VillaniOptimalTransport} for details.
It follows that
\begin{equation}
\label{eq:ps-neg-conditional-t2}
- \frac {(1+\varepsilon)}2 \sum_{k \in [N]}
\Expect [ \langle W_{-,*},
(\nu^{N,\sigma(k) | \sigma([k-1])}_{\bm{X}^{\sigma([k-1])}})^{\otimes 2}\rangle
] + H(\nu^N | m_*^{\otimes N})
\geqslant \frac{\lambda}{2} W_2^2(\nu,m_*^{\otimes N}).
\end{equation}
It remains to find a chain rule for the energy term on the left.
This is formulated in the following lemma.

\begin{lem}
\label{lem:conditional-energy}
Let $\alpha \in [0,1/2]$. Denote
\[
A = \frac 1N \sum_{k \in [N]}
\Expect [ W_{-,*}(X^i, X^i) ].
\]
Then it holds that
\begin{multline}
\label{eq:conditional-energy}
\frac 1{N!}\sum_{\sigma \in S_N}
\sum_{k \in \llbracket 2,N\rrbracket}
\Expect [\langle W_{-,*},
(\nu^{N,\sigma(k) | \sigma([k-1])}_{\bm{X}^{\sigma([k-1])}}
)^{\otimes 2}\rangle] \\
\geqslant
\begin{cases}
\bigl( \frac{2}{1 + \alpha} - \frac{1}{1 + 2\alpha} \bigr)
N \Expect[\langle W_{-,*}, \mu_{\bm{X}}^{\otimes 2}\rangle]
- \bigl(\frac{1}{2\alpha}+3\bigr) A, & \alpha \in (0,1/2], \\
N \Expect[\langle W_{-,*}, \mu_{\bm{X}}^{\otimes 2}\rangle]
- (\log N + 3) A, & \alpha = 0.
\end{cases}
\end{multline}
\end{lem}

\begin{proof}[Proof of Lemma~\ref{lem:conditional-energy}]
For $N = 1$, we have
\[
\Expect [ \langle W_{-,*}, \mu_{\bm{X}}^{\otimes 2}\rangle ]
= \Expect_{X^1\sim\nu^1}
[ \langle W_{-,*}, \delta_{X^1} \otimes \delta_{X^1} \rangle]
= A.
\]
The positivity of $W_{-,*}$ implies that $A \geqslant 0$.
The right hand side of the claim of the lemma is therefore negative.
The summation on the left being null,
the inequality \eqref{eq:conditional-energy} automatically holds.

In the following we focus on the case $N \geqslant 2$.
For simplicity of notation, denote
\begin{align*}
W_{i,j} &= W_{-,*}(X^i, X^j) \\
B &= \frac{1}{N(N-1)} \sum_{\substack{i,j \in [N] \\ i \neq j}}
\Expect [ W_{i,j} ].
\end{align*}
The energy part in the modulated free energy
$\mathcal F^N(\nu^N) - N\mathcal F^N(m_*)$ thus writes
\[
N \Expect [ \langle W_{-,*}, \mu_{\bm{X}}^{\otimes 2}\rangle]
= \frac 1N \sum_{i,j \in [N]} \Expect [W_{i,j}] = A + (N-1) B.
\]
The positivity of $W_{-}$ implies that
$A \geqslant 0$ and $A + (N-1)B \geqslant 0$.
Reordering the indices by a permutation $\sigma \in S_N$,
we obtain, for each $k \in \llbracket 2,N\rrbracket$,
\begin{IEEEeqnarray*}{rCl}
\sum_{i \in [k-1]} \Expect [W_{\sigma(i), \sigma(k)}]
&=& \Expect [ \langle W_{-,*},
(\delta_{X^{\sigma(1)}} + \cdots + \delta_{X^{\sigma(k-1)}})
\otimes \delta_{X^{\sigma(k)}} \rangle] \\
&=& \Expect [ \langle W_{-,*},
(\delta_{X^{\sigma(1)}} + \cdots + \delta_{X^{\sigma(k-1)}})
\otimes \nu^{N,\sigma(k) | \sigma([k-1])}_{\bm{X}^{\sigma([k-1])}} \rangle] \\
&\leqslant&
\frac{c_k}{2(k-1)}\Expect [\langle W_{-,*},
(\delta_{X^{\sigma(1)}} + \cdots + \delta_{X^{\sigma(k-1)}})^{\otimes 2}\rangle]
\\
&&\adjustbin+ \frac{k-1}{2c_k}
\Expect [\langle W_{-,*},
(\nu^{N,\sigma(k) | \sigma([k-1])}_{\bm{X}^{\sigma([k-1])}})^{\otimes 2}\rangle
],
\end{IEEEeqnarray*}
where the last Cauchy--Schwarz inequality relies on the positivity of
$W_{-,*}$ and
\[
c_k \coloneqq \biggl( \frac{k-1}{N-1} \biggr)^{\!\alpha}.
\]
Summing over $k \in \llbracket 2, N\rrbracket$ leads to
\begin{multline*}
\sum_{k \in \llbracket 2,N\rrbracket}
\Expect [\langle W_{-,*},
(\nu^{N,\sigma(k) | \sigma([k-1])}_{\bm{X}^{\sigma([k-1])}})^{\otimes 2}
\rangle] \\
\geqslant \sum_{k \in \llbracket 2,N\rrbracket}
\biggl(\frac{2c_k}{k-1} \sum_{i \in [k-1]}
\Expect [W_{\sigma(i), \sigma(k)}]
- \frac{c^2_k}{(k-1)^2} \sum_{i,j \in [k-1]}
\Expect [W_{\sigma(i), \sigma(j)}]\biggr).
\end{multline*}
Averaging the inequality over permutations, we find
\begin{IEEEeqnarray}{rCl}
\IEEEeqnarraymulticol{3}{l}
{\frac 1{N!}\sum_{\sigma \in S_N}
\sum_{k \in \llbracket 2,N\rrbracket}
\Expect [\langle W_{-,*},
(\nu^{N,\sigma(k) | \sigma([k-1])}_{\bm{X}^{\sigma([k-1])}})^{\otimes 2}
\rangle]}
\notag \\
\quad&\geqslant&
\frac 1{N!}\sum_{\sigma \in S_N}
\sum_{k \in \llbracket 2,N\rrbracket}
\biggl(\frac{2c_k}{k-1} \!\sum_{i \in [k-1]}\!
\Expect [W_{\sigma(i), \sigma(k)}]
- \frac{c^2_k}{(k-1)^2} \!\!\sum_{i,j \in [k-1]}\!\!
\Expect [W_{\sigma(i), \sigma(j)}]\biggr) \notag \\
\quad&=&
\sum_{k \in \llbracket 2,N\rrbracket}
\biggl(2c_k - \frac{(k-2)c_k^2}{k-1}\biggr) B
- \sum_{k \in \llbracket 2, N\rrbracket} \frac{c_k^2}{k-1}A
\label{eq:conditional-energy-last-step} \\
&=& \frac 1{N-1}\sum_{k \in \llbracket 2,N\rrbracket}
\biggl(2c_k - \frac{(k-2)c_k^2}{k-1}\biggr) \bigl(A + (N-1)B\bigr) \notag \\
&&\adjustbin- \sum_{k \in \llbracket 2, N\rrbracket}
\biggl(\frac{c_k^2}{k-1} + \frac{2c_k}{N-1}
- \frac{(k-2)c_k^2}{(N-1)(k-1)}\biggr)A. \notag
\end{IEEEeqnarray}
What is left is therefore to control the coefficients of $A + (N-1)B$ and $A$,
and we do it by considering the two cases $\alpha \in (0,1/2]$ and $\alpha = 0$
separately.

In the case $\alpha \in (0,1/2]$,
the coefficient of $A + (N-1)B$ satisfies
\begin{align*}
\frac 1{N-1}\sum_{k \in \llbracket 2, N\rrbracket}
\biggl(2c_k - \frac{(k-2)c_k^2}{k-1}\biggr)
&\geqslant
\frac 1{N-1}\sum_{k \in \llbracket 2, N\rrbracket}
(2c_k - c_k^2) \\
&\geqslant
\int_{0}^{1} (2x^{\alpha} -x^{2\alpha}) \dd x
= \frac 2{1+\alpha} - \frac{1}{1+2\alpha}
\end{align*}
and the coefficient of $A$ satisfies
\begin{IEEEeqnarray*}{rCl}
\IEEEeqnarraymulticol{3}{l}
{\sum_{k \in \llbracket 2, N\rrbracket}
\biggl(\frac{c_k^2}{k-1} + \frac{2c_k}{N-1}
- \frac{(k-2)c_k^2}{(N-1)(k-1)}\biggr)} \\
\quad&\leqslant&
\sum_{k \in \llbracket 2, N\rrbracket}
\biggl(\frac{c_k^2}{k-1} + \frac{2c_k}{N-1}\biggr)
\leqslant \frac{1}{(N-1)^{2\alpha}}
\Bigl( 1 + \sum_{k \in \llbracket 2, N-1\rrbracket} k^{2\alpha-1}\Bigr) + 2
\leqslant \frac{1}{2\alpha}+3.
\end{IEEEeqnarray*}
In the case $\alpha = 0$, the coefficient of $A+(N-1)B$ satisfies
\[
\frac 1{N-1}\sum_{k \in \llbracket 2, N\rrbracket}
\biggl(2c_k - \frac{(k-2)c_k^2}{k-1}\biggr)
\geqslant 1
\]
and the coefficient of $A$ satisfies
\[
\sum_{k \in \llbracket 2, N\rrbracket}
\biggl(\frac{c_k^2}{k-1} + \frac{2c_k}{N-1}
- \frac{(k-2)c_k^2}{(N-1)(k-1)}\biggr)
\leqslant \sum_{k \in \llbracket 2, N-1\rrbracket} \frac 1k
+ 3
\leqslant \log N + 3.
\]
Inserting the coefficient bounds
into \eqref{eq:conditional-energy-last-step}
proves the lemma.
\end{proof}

Having disposed of the proof of the lemma, we now return
to lower bounding the modulated free energy
\[
- \frac{N}{2} \Expect [\langle W_{-,*}, \mu_{\bm{X}}^{\otimes 2}\rangle]
+ H(\nu^N | m_*^{\otimes N}).
\]
For $\varepsilon \in (0,1/2]$,
we apply the lemma with $\alpha = \sqrt{\varepsilon/2}$
and take the average of \eqref{eq:ps-neg-conditional-t2}
over $\sigma \in S_N$ to obtain
\begin{multline*}
-\biggl(\frac{2}{1+\alpha} - \frac{1}{1+2\alpha}\biggr)
\frac{(1+\varepsilon)N}{2}
\Expect[\langle W_{-,*}, \mu_{\bm{X}}^{\otimes 2}\rangle]
+ H(\nu^N|m_*^{\otimes N}) \\
\geqslant \frac{\lambda}{2} W_2^2(\nu^N, m_*^{\otimes N})
- \frac 32 \biggl(\frac{1}{4\alpha}+\frac 32\biggr) A.
\end{multline*}
Our choice of $\alpha$ ensures that
\[
\biggl(\frac{2}{1+\alpha} - \frac{1}{1+2\alpha}\biggr) (1+\varepsilon)
= \frac{(1+3\alpha)(1+\varepsilon)}{1+3\alpha+2\alpha^2}
\geqslant \frac{1+\varepsilon}{1+2\alpha^2} = 1.
\]
It follows that
\[
- \frac{N}{2} \Expect[\langle W_{-,*}, \mu_{\bm{X}}^{\otimes 2}\rangle]
+ H(\nu^N|m_*^{\otimes N})
\geqslant \frac{\lambda}{2} W_2^2(\nu^N, m_*^{\otimes N})
- \frac 32 \biggl(\frac{1}{4\alpha}+\frac 32\biggr) A.
\]
For $\varepsilon = 0$, we apply the lemma with $\alpha = 0$
and similarly obtain
\[
-\frac{N}{2}
\Expect[\langle W_{-,*}, \mu_{\bm{X}}^{\otimes 2}\rangle]
+ H(\nu^N|m_*^{\otimes N})
\geqslant \frac{\lambda}{2} W_2^2(\nu^N, m_*^{\otimes N})
- \frac{\log N+3}{2}A.
\]

To complete the lower bound, it remains to control the diagonal term $A$.
Let $\bm{Y}$ be a random variable that is $W_2$-optimally coupled with
$\bm{X}$, i.e.,
\[
\sum_{k \in [N]} \Expect [ \lvert X^k - Y^k \rvert^2 ]
= W_2^2(\nu^N, m_*^{\otimes N}).
\]
Using the $L^\infty$ bound on $\nabla_{1,2}^2W_-$
and Cauchy--Schwarz, we obtain
\begin{IEEEeqnarray*}{rCl}
A &=& \frac 1N \sum_{k \in [N]} \Expect
[\langle W_-, (\delta_{X^k} - m_*)^{\otimes 2}\rangle] \\
&\leqslant& \frac {L_W^-}N \sum_{k \in [N]}
\Expect [W_2^2(\delta_{X^k}, m_*)] \\
&\leqslant& \frac{2L^-_W}{N} \sum_{k \in [N]}
\Expect[W_2^2(\delta_{X^k}, \delta_{Y^k}) + W_2^2(\delta_{Y^k}, m_*)] \\
&\leqslant& \frac{2L_W^-}{N} W_2^2(\nu^N, m_*^{\otimes N})
+ 4L_W^- \Var m_*.
\end{IEEEeqnarray*}

\proofstep{Step 3: Upper bound of
$\mathcal F^N(m^N_*|m_*)$}
Again let $\bm{Y}$ be distributed according to the law $m_*^{\otimes N}$.
Since $m^N_*$ minimizes the $N$-particle free energy $\mathcal F^N$,
\begin{IEEEeqnarray*}{rCl}
\mathcal F^N(m^N_* | m_*)
&=& \mathcal F^N(m^N_*) - N\mathcal F(m_*) \\
&\leqslant& \mathcal F^N(m_*^{\otimes N}) - N\mathcal F(m_*) \\
&=& \frac{N}{2}\Expect[ \langle W,
(\mu_{\bm{Y}} - m_*)^{\otimes 2}\rangle] \\
&\leqslant& \frac{N}{2}\Expect[ \langle W_+,
(\mu_{\bm{Y}} - m_*)^{\otimes 2}\rangle] \\
&=& \frac{1}{2N} \sum_{k \in [N]}
\Expect[ \langle W_+, (\delta_{Y^k} - m_*)^{\otimes 2}\rangle] \\
&\leqslant& L_W^+ \Var m_*.
\end{IEEEeqnarray*}
Combining this with the inequalities from the last step yields,
for $\varepsilon > 0$,
\begin{multline*}
H(\nu^N | m_*^N)
\geqslant \biggl( \frac{\lambda}{2}
- \frac{3(1/\sqrt{2\varepsilon}+3) L_W^-}{2N}
\biggr) W_2^2(\nu^N, m_*^{\otimes N}) \\
- \biggl(3\Bigl(\frac{1}{\sqrt{2\varepsilon}}+3\Bigr)L^-_W
+ L^+_W\biggr) \Var m_*;
\end{multline*}
and for $\varepsilon = 0$,
\begin{multline*}
H(\nu^N | m_*^N)
\geqslant
\biggl( \frac{\lambda}{2}
- \frac{(\log N + 3)L^-_W}{N} \biggr)
W_2^2(\nu^N, m_*^{\otimes N}) \\
- \bigl(2(\log N+3) L^-_W + L^+_W\bigr) \Var m_*,
\end{multline*}
proving the defective $\TTE2$ inequality \eqref{eq:ps-defective-t2}.

\proofstep{Step 4: Tightening and marginalization of $\TTE1$ inequality}
Having shown the defective $\TTE2$ inequality \eqref{eq:ps-defective-t2},
we proceed to the proof of the $\TTE1$ inequality \eqref{eq:ps-t1}.
Substituting $\nu^N = m^N_*$ into \eqref{eq:ps-defective-t2} gives
\[
W_2^2(m^N_*, m_*^{\otimes N}) \leqslant \frac{2\delta_N}{\lambda_N}.
\]
By Cauchy--Schwarz,
\[
W_2^2(\nu^N, m_*^{\otimes N}) \geqslant \frac 12 W_2^2(\nu^N, m^N_*)
- W_2^2(m^N_*, m_*^{\otimes N})
\geqslant \frac 12W_2^2(\nu^N, m^N_*)
- \frac{2\delta_N}{\lambda_N}.
\]
Hence from \eqref{eq:ps-defective-t2} we deduce that
\[
W_2^2(\nu^N,m^N_*) \leqslant \frac{4}{\lambda_N}
\bigl( H(\nu^N|m^N_*) + \delta_N\bigr) + \frac{4\delta_N}{\lambda_N}
= \frac{4}{\lambda_N}
\bigl( H(\nu^N|m^N_*) + 2\delta_N\bigr).
\]
Dominating the $W_1$ distance by $W_2$,
we obtain the defective $\TTE1$ inequality \eqref{eq:defective-t1}
for the $N$-particle Gibbs measure $m^N_*$.
Applying the tightening of Theorem~\ref{thm:t1-tightening}
yields the $\TTE1$ inequality \eqref{eq:ps-t1}.
Finally, for the $\TTE1$ inequalities of the marginal distributions,
it suffices to consider the equivalent formulation of
\textcite{BobkovGoetzeExponential}
(see also the beginning of the proof of
Theorem~\ref{thm:t1-tightening})
and take test functions that only depend on the corresponding coordinates.
This concludes the proof of the theorem for the general case.
The assertion for the special case of a flat convex $W$
follows by setting $W_- = 0$ with arbitrary $\varepsilon$.

\proofstep{Step 5: Large deviation estimate}
By the Donsker--Varadhan variational principle,
\begin{align*}
\MoveEqLeft
\log \int_{\R^{Nd}} \exp \biggl(
- \frac N2 \langle W, (\mu_{\vect x} - m_*)^{\otimes 2}\rangle\biggr)
m_*^{\otimes N}(\dd\vect x)\\
&= \sup_{\nu^N \in \mathcal P_2(\R^{Nd})}
\biggl(  - \frac N2 \Expect_{\vect X \sim \nu^N} [\langle W,
(\mu_{\vect X} - m_*)^{\otimes 2}] - H(\nu^N | m_*^{\otimes N}) \biggr) \\
&= - \inf_{\nu^N \in \mathcal P_2(\R^{Nd})}
\mathcal F^N(\nu^N | m_*)
\end{align*}
Combining the estimates in Step~2 for the case $\varepsilon \in (0,1/2]$,
we obtain
\begin{align*}
\mathcal F^N(\nu^N | m_*)
&\geqslant
- \frac N2 \Expect [\langle W_{-,*}, \mu_{\vect X}^{\otimes 2}\rangle]
+ H(\nu^N | m_*^{\otimes N}) \\
&\geqslant \frac{\lambda_N}{2} W_2^2(\nu^N, m_*^{\otimes N})
- 3 \biggl( \frac{1}{\sqrt{2\varepsilon}} + 3\biggr) L_W^- \Var m_*.
\end{align*}
Since $\lambda_N \geqslant 0$, this completes the proof of the corollary.
\qed

\section{Proof of Proposition~\ref{prop:exm}}

We prove the proposition by verifying
the conditions of Theorem~\ref{thm:sharp} in the three cases separately.

\proofstep{Case 1: Flat semi-convexity with bounded force}
The uniform log-Sobolev inequality for $\Pi[m]$
implies immediately that
\[
2\rho_0 H(m | m_*) \leqslant I(m|m_*).
\]
We proceed to show the non-linear log-Sobolev inequality.
Applying the uniform log-Sobolev for $\Pi[m]$ yields
\[
2\rho_0 H(m | \Pi[m]) \leqslant I(m | \Pi[m]).
\]
The left hand side satisfies
\begin{IEEEeqnarray*}{rCl}
H(m | \Pi[m])
&=& H(m | m_*) + \langle W, (m - m_*)^{\otimes 2}\rangle \\
&&\adjustbin + \log \int_{\X}
e^{-\langle W(x,\cdot), m - m_*\rangle} m_*(\dd x)
+ \int_{\X} \langle W(x,\cdot), m - m_*\rangle m_*(\dd x) \\
&\geqslant& H(m|m_*) - \langle W_-, (m - m_*)^{\otimes 2}\rangle,
\end{IEEEeqnarray*}
where the last inequality is due to the positivity of $W_+$
and the convexity of $t \mapsto e^{-t}$.
The energy term satisfies
\begin{IEEEeqnarray*}{rCl}
\langle W_-, (m - m_*)^{\otimes 2}\rangle
&=& \int_{\X}\biggl(\int_{\X}W_-(x,y) (m-m_*)(\dd x)\biggr) (m-m_*)(\dd y) \\
&\leqslant& \int_{\X} \sup_{x\in\X}\lVert \nabla_1W_-(x,y)\rVert_{L^\infty}
W_1(m,m_*) (m-m_*)(\dd y) \\
&\leqslant& M_W^- W_1(m,m_*) \sqrt{2H(m|m_*)} \\
&\leqslant& \frac{2M_W^-}{\sqrt{\rho_0}} H(m|m_*).
\end{IEEEeqnarray*}
The non-linear log-Sobolev follows:
\[
I(m|\Pi[m]) \geqslant 2\rho_0 \biggl( 1 - \frac{2M_W^-}{\sqrt{\rho_0}}\biggr)
H(m|m_*).
\]
The force being bounded in this case,
the transport inequality is always verified
with $\gamma = 2 M_W$ according to Pinsker
and the square integrability with $M = 4 M_W^2$.

\proofstep{Case 2: Flat semi-convexity with Lipschitz force}
For the linear and non-linear log-Sobolev inequalities
we argue as in the previous case.
The only difference is that we control the concave energy term by
\[
\langle W_-, (m - m_*)^{\otimes 2}\rangle
\leqslant L^-_W W_2^2(m, m_*)
\leqslant \frac{2L^-_W}{\rho_0} H(m | m_*).
\]
This leads to the non-linear log-Sobolev:
\[
I(m | \Pi[m]) \geqslant 2\rho_0
\biggl( 1 - \frac{2L^-_W}{\rho_0}\biggr) H(m | m_*).
\]

The $\TTE1$ inequality for $m_*$ being a consequence of
the $\rho_0$-log-Sobolev,
we proceed to show that a $\TTE1$ inequality holds
for $m^{N,k|[k-1]}_{*,\x^{[k-1]}}$ uniformly.
Our strategy is as follows.
First, we note that
$m^{N,k|[k-1]}_{*,\x^{[k-1]}} = \Law(X^k | \vect X^{[k-1]})$
is the $1$-marginal of the conditional law
\[
\Law (\vect{X}^{\llbracket k,N\rrbracket} |
\bm{X}^{[k-1]} = \x^{[k-1]})
\coloneqq \Law ( X^k, \ldots, X^N |
\bm{X}^{[k-1]} = \x^{[k-1]})
\]
and it suffices to establishes a $\TTE1$ inequality for this conditional law.
Next, we observe that this conditional law is a Gibbs measure
associated with an energy functional parametrized by the configuration
of the first $k-1$ particles.
Therefore, it remains only to verify the conditions
of Theorem~\ref{thm:defective-t2} uniformly
for these parametrized energy functionals.

Fix $k \in [N]$ and $\x^{[k-1]} \in \X^{k-1}$.
Consider the mean field energy functional $F_{\x^{k-1}}$
\[
F_{\x^{[k-1]}} (m)
= \langle V, m\rangle
+ \frac{k-1}{N} \langle W, \mu_{\x^{[k-1]}} \otimes m\rangle
+ \frac{(N-k+1)}{2N}
\langle W,m^{\otimes 2}\rangle
\]
parameterized by $\x^{[k-1]} \in \R^{(k-1)d}$.
Define the parameterized potentials
\begin{align*}
V_{\x^{[k-1]}}(y) &= V(y) + \frac 1N \sum_{i \in [k-1]} W(x^i, y), \\
W_{\x^{[k-1]}}(y,z) &= \frac{N-k+1}{N} W(y,z)
= \frac{N-k+1}{N} \bigl( W_+(y,z) - W_-(y,z)\bigr).
\end{align*}
Thus, the parameterized energy functional satisfies
\[
F_{\x^{[k-1]}} (m)
= \langle V_{\x^{[k-1]}}, m\rangle
+ \frac 12 \langle W_{\x^{[k-1]}}, m^{\otimes 2}\rangle,
\]
and falls within the scope of the paper,
whereas the original functional is given by
\[
F(m)
= \langle V, m\rangle
+ \frac 12 \langle W, m^{\otimes 2}\rangle.
\]
The construction of $F_{\x^{[k-1]}}$ ensures that
\[
N F(\mu_{\x^{[N]}}) -
(N-k+1) F_{\x^{[k-1]}}
(\mu_{\x^{\llbracket k, N\rrbracket}}),
\]
where $\x^{\llbracket k,N\rrbracket} \coloneqq (x^{k},\ldots,x^N)$,
is only a function of $\x^{[k-1]}$.
Consequently, the $(N-k+1)$-particle invariant measure
associated to $F_{\x^{[k-1]}}$,
being proportional to
\[
\exp\bigl(
- (N-k+1) F_{\x^{[k-1]}}
(\mu_{\x^{\llbracket k, N\rrbracket}})
\bigr) \dd x^k \cdots \dd x^N
\propto m^N_*(\x^{[N]}) \dd x^k \cdots \dd x^N
\]
is nothing but the conditional law
\[
m^{N,\llbracket k, N\rrbracket | [k-1]}_{*,\x^{[k-1]}}
\coloneqq \Law (\bm{X}^{\llbracket k,N\rrbracket} |
\bm{X}^{[k-1]} = \x^{[k-1]}),
\]
where $\bm{X} \coloneqq (X^1,\ldots,X^N)$ is distributed
according to $m^N_*$.
Therefore, once we verify conditions of Theorem~\ref{thm:defective-t2}
for the potentials $V_{\x^{[k-1]}}$, $W_{\x^{[k-1]}}$
uniformly in $k$ and $\x^{[k-1]}$,
the desired $\TTE1$ inequality for $m^{N,k|[k-1]}_{*,\x^{[k-1]}}$ follows
and the transport condition for Theorem~\ref{thm:sharp} is satisfied.
This is how we proceed in the following.

The mean field Gibbs measure $m_{*,\x^{[k-1]}}$ associated to
$F_{\x^{[k-1]}}$ should solve, by definition, the fixed point equation
\[
m_{*,\x^{[k-1]}}(\dd y) \propto \exp
\bigl( - V_{\x^{[k-1]}}(y)
- \langle W_{\x^{[k-1]}}(y,\cdot),m\rangle \bigr) \dd y.
\]
Substituting the definitions of $V_{\x^{[k-1]}}$ and $W_{\x^{[k-1]}}$,
we see that this is equivalent to the following:
\[
m_{*,\x^{[k-1]}} = \Pi\biggl[
\frac{k-1}{N} \mu_{\x^{[k-1]}} + \frac{N-k+1}{N}
m_{*,\x^{[k-1]}}\biggr].
\]
Hence, if the measure $m_{*,\x^{[k-1]}}$ exists,
it satisfies a $\rho_0$-log-Sobolev inequality by assumption.
To establish the existence of the measure,
we observe that it can be obtained as the minimizer of the following functional:
\[
\nu \mapsto
\frac{k-1}{N} \langle W_*, \mu_{\x^{[k-1]}} \otimes \nu\rangle
+ \frac{N-k+1}{2N} \langle W_*, \nu^{\otimes 2}\rangle
+ H(\nu|m_*),
\]
which, under our assumptions, is bounded from below, coercive
and lower semi-continuous with respect to the weak topology.
The log-Sobolev inequality for $m_{*,\x^{[k-1]}}$ implying
the $\TTE2$ for the same measure,
it follows that for all $\nu \in \mathcal P_2(\R^d)$,
\begin{multline*}
H(\nu | m_{*,\x^{[k-1]}})
- \frac {3(N-k+1)}{4N} \langle W_-, (\nu - m_{*,\x^{[k-1]}})^{\otimes 2}\rangle
\\
\geqslant
\frac{\rho_0}{2} W_2^2(\nu, m_{*,\x^{[k-1]}})
- \frac {3}{4} L^-_W W_2^2(\nu, m_{*,\x^{[k-1]}}),
\end{multline*}
verifying $\TTE2$ inequality \eqref{eq:mf-neg-t2} with
\[
\lambda = \rho_0 - \frac{3}{2} L^-_W,\qquad
\varepsilon = \frac 12.
\]
Applying Theorem~\ref{thm:defective-t2}
to the parameterized interaction $F_{\x^{[k-1]}}$
yields a $\TTE2$ inequality \eqref{eq:ps-t1}
for $m^{N,\llbracket k, N\rrbracket | [k-1]}_{*,\x^{[k-1]}}$
with the constants $\lambda_N$, $\delta_N$ as claimed in the proposition.
The measure $m^{N,k|[k-1]}_{*,\x^{[k-1]}}$ being
the $1$-marginal of $m^{N,\llbracket k, N\rrbracket | [k-1]}_{*,\x^{[k-1]}}$,
it also satisfies the $\TTE1$ inequality by the last claim of the theorem.
We have thus verified the transport condition with
\[
\gamma = \frac{64(1+\delta_N)^2L_W^2}{\lambda_N}.
\]

Finally, for the square integrability, note that both integrals in the condition
can be upper bounded by
\[
(L^+_W+L^-_W)^2 \Expect_{X^1\sim m^{N,1}_*}[W_2^2(\delta_{X^1}, m_*)],
\]
while the expectation satisfies
\[
\Expect_{X^1\sim m^{N,1}_*}[W_2^2(\delta_{X^1}, m_*)]
\leqslant \frac{2}{N} W_2^2(m^N_*, m_*^{\otimes N}) + 4 \Var m_*.
\]
by the end of Step~2 of the proof of Theorem~\ref{thm:defective-t2}.
Taking $k = 1$ in the argument of the previous paragraph
yields the defective $\TTE2$ inequality \eqref{eq:ps-defective-t2}
for $m^N_*$, and hence,
\[
W_2^2(m^N_*, m_*^{\otimes N}) \leqslant \frac{2\delta_N}{\lambda_N}.
\]
By Poincaré, $\Var m_* \leqslant d/\rho_0$.
We can choose the integrability constant as follows:
\[
M = 4(L^+_W+L^-_W)^2
\biggl( \frac{\delta_N}{N\lambda_N} + \frac{d}{\rho_0}\biggr).
\]

\proofstep{Case 3: Displacement convexity}
This is the case studied by \textcite{CMVKinetic}
where they show that the mean field free energy functional
is $\kappa_V$-convex along geodesics in $W_2$.
By Theorem~2.1 therein,
or alternatively by the general results from \cite{AGSGradientFlows},
\[
I(m|\Pi[m]) \geqslant \kappa_V^2 W_2^2(m, m_*).
\]
As $ - \nabla^2 \log m_* = \nabla^2V + \nabla^2W\star m_*
\succcurlyeq \kappa_V$,
the measure $m_*$ satisfies a $\kappa_V$-log-Sobolev inequality
by the Bakry--Émery criterion \cite{BakryEmeryHypercontractives}.
It follows that
\[
2\kappa_V H(m|m_*)
\leqslant I(m|m_*)
\leqslant 2I(m | \Pi[m])
+ 2\int_{X} \biggl| \nabla \log \frac{\Pi[m]}{\Pi[m_*]}\biggr|^2
\dd m,
\]
where the last term satisfies
\[
\biggl| \nabla \log \frac{\Pi[m]}{\Pi[m_*]}\biggr|^2
= \lvert\langle \nabla W \star (m - m_*) \rangle\rvert^2
\leqslant L_W^2 W_2^2(m, m_*).
\]
The non-linear log-Sobolev inequality is thus obtained:
\[
2\kappa_V H(m|m_*)
\leqslant 2I(m|\Pi[m])
+ 2L_W^2 W_2^2(m,m_*)
\leqslant 2\biggl(1 + \frac{L_W^2}{\kappa_V^2}\biggr)
I(m | \Pi[m]).
\]

For the transport inequality, we argue as in
\cite[Proof of Corollary~2.10]{LackerQuantitative}.
By direct computations, the Gibbs measure $m^N_*$ is $\kappa_V$-log-concave
(see e.g.\ \textcite{MalrieuLSI}).
The strong log-concavity being stable under marginalization and conditioning,
the conditional measure $m^{N,k|[k-1]}_{*,\x^{[k-1]}}$
is also $\kappa_V$-log-concave.
It follows that the transport inequality holds with
\[
\gamma = \frac{2L_W^2}{\kappa_V}.
\]

Finally, using the log-Sobolev inequality for $m^N_*$, we find
\begin{align*}
\kappa_V^2 W_2^2(m^N_*, m_*^{\otimes N})
&\leqslant I(m_*^{\otimes N} | m_*^N) \\
&= \frac{1}{N^2} \sum_{i \in [N]}
\Expect_{\bm{Y} \sim m_*^{\otimes N}}
\Bigl[\Bigl|\sum_{j \in [N]}\nabla_1 W_*(Y^i, Y^j)\Bigr|^2\Bigr] \\
&= \frac 1N \sum_{i \in [N]}
\Expect [\lvert \nabla_1 W_*(Y^1,Y^i)\rvert^2] \\
&\leqslant L_W^2 \Expect[W_2^2(\delta_{Y^1}, m_*)]
\leqslant \frac{2L_W^2d}{\kappa_V},
\end{align*}
where we recall that $W_*$ is the reduced potential
defined by \eqref{eq:def-w-reduced}.
It follows again by the Cauchy--Schwarz argument used in the previous case that
\[
\int_{\X} W_2^2(\delta_{x^1}, m_*) m^{N,1}_*(\dd x)
\leqslant \frac{2}{N} W_2^2(m^N_*, m_*^{\otimes N})
+ 4 \Var m_*
\leqslant \frac{4L_W^2d}{\kappa_V^3N} + \frac{4d}{\kappa_V}.
\]
The integrability condition is satisfied with
\[
M = 4L_W^2 \biggl( 1 + \frac{L_W^2}{\kappa_V^2N} \biggr) \frac{d}{\kappa_V}.
\tag*{\qed}
\]

\section{Proof of Proposition~\ref{prop:curie-weiss}}

The proof will be divided into three steps.

\proofstep{Step 1: Properties of the mapping $\Pi$}
In this part we develop some elementary results on
the mapping $\Pi$ associated to the mean field energy
\[
F(m) = \int_{\R} V(x) m(\dd x)
- \frac{J}{2} \biggl( \int_{\R} x m(\dd x)\biggr)^{\!2}.
\]
Let us define
\begin{align*}
p(m) &= \int_{\R} x m(\dd x), \\
\pi[h](\dd x) &= \frac{e^{Jhx-V(x)}\dd x}
{\int_{\R} e^{Jhy-V(y)} \dd y}.
\end{align*}
The mapping $\Pi$ decomposes as follows:
\[
\Pi[m] = \pi[p(m)].
\]
Moreover, the fixed point relation
\[
m = \Pi[m]
\]
holds if and only if $m = \pi[h]$ for some $h \in \R$ solving
\[
h = p (\pi[h]).
\]
Denote by $f$ the composed function $p \circ \pi \colon \R \to \R$,
which writes
\[
f(h) = \frac{\int_{\R} x e^{Jhx-V(x)} \dd x}
{\int_{\R} e^{Jhx-V(x)} \dd x}.
\]
It follows immediately that $f$ is odd.
The derivative writes
\[
f'(h) = J \frac{\int_{\R} x^2 e^{Jhx-V(x)} \dd x}
{\int_{\R} e^{Jhx-V(x)} \dd x}
- J \biggl(\frac{\int_{\R} x e^{Jhx-V(x)} \dd x}
{\int_{\R} e^{Jhx-V(x)} \dd x}\biggr)^{\!2},
\]
which is strictly positive everywhere.
Hence $f$ is strictly increasing.
As $f$ satisfies $f(+\infty) = +\infty$ and $f(-\infty)=-\infty$,
it defines a $\mathcal C^1$ bijection
in $\R$ with a $\mathcal C^1$ inverse, denoted by $f^{-1}$.
Taking $h=0$ in the expression for $f'(h)$ gives
\[
f'(0) = J \frac{\int_{\R} x^2 e^{Jhx-V(x)} \dd x}
{\int_{\R} e^{Jhx-V(x)} \dd x} = \frac{J}{\Jc} < 1
\]
In addition, the Griffiths--Hurst--Sherman inequality implies that
\[
f''(h) \begin{cases}
\leqslant 0, & h > 0, \\
\geqslant 0, & h < 0.
\end{cases}
\]
See e.g.\ Corollary~4.3.4 and the remark that follows
in the Glimm--Jaffe book \cite{GlimmJaffeQuantum}.
We therefore have $f'(h) \leqslant f'(0) < 1$ for all $h \in \R$.
According to the contraction mapping theorem,
the equation
\[
h = f(h)
\]
admits the unique solution $h = 0$.
Equivalently, the centered measure
\[
m_*(\dd x) = \pi[0](\dd x) = \frac{e^{-V(x)} \dd x}
{\int_{\R} e^{-V(y)}\dd y}
\]
is the only fixed point of $\Pi$.

\proofstep{Step 2: Log-Sobolev inequality}
We first show that images of the mapping $\Pi$
admit a uniform log-Sobolev inequality.
By the previous step, it is equivalent to show a uniform log-Sobolev
for $\pi[h]$, whose log-density writes
\[
- \log \pi[h] (x) = V(x) + Jhx.
\]
Note that $(Jhx)'' = 0$ and
\[
V''(x) \geqslant 3\theta x^2 + \sigma.
\]
For $\sigma \geqslant 1$,
Bakry--Émery \cite{BakryEmeryHypercontractives} implies
a $\sigma$-log-Sobolev inequality for $\pi[h]$.
For $\sigma < 1$, we define
\[
x_0 = \sqrt{\frac{(1-\sigma)^2}{3\theta}}.
\]
Observe that for $\lvert x \rvert \geqslant x_0$,
we have $V''(x) \geqslant 1$.
Construct the decomposition for the confinement:
\[
V(x) = V(x) \1_{\lvert x\rvert \geqslant x_0}
+ \biggl( \frac{x^2}{2} - \frac{x_0^2}{2} + V(x_0)
\biggr)\1_{\lvert x\rvert < x_0} + V_0(x),
\]
the $V_0$ part being supported on $B(0,x_0)$ with bounded oscillation:
\[
\Osc V_0 \leqslant \frac{\theta}{4}x_0^4 + \frac{(1-\sigma)}{2}x_0^2
\leqslant \frac{7(1-\sigma)^2}{36\theta}.
\]
By the Bakry--Émery criterion
and the Holley--Stroock perturbation lemma \cite{HolleyStroockLSI},
$\pi[h]$ satisfies a log-Sobolev inequality with the following constant:
\[
\exp \biggl( - \frac{7(1-\sigma)^2}{36\theta} \biggr).
\]
Hence it satisfies a $\rho_0$-log-Sobolev inequality
as stated in the proposition.
In particular, $m_*$ satisfies the same $\rho_0$-inequality.

Now we proceed to verify the non-linear log-Sobolev inequality.
By the $\rho_0$-log-Sobolev for $\Pi[m]$,
\[
I(m | \Pi[m]) \geqslant 2\rho_0 H(m | \Pi[m]).
\]
Denote $h = p(m) = \int_{\R} x m(\dd x)$.
The relative entropy on the right writes
\begin{align*}
H(m | \Pi[m])
&= H(m | m_*) + \int_{\R} \log \frac{m_*}{\Pi[m]} \dd m \\
&= H(m | m_*) - Jh^2
+ \log \int_{\R} e^{Jhx - V(x)} \dd x
- \log \int_{\R} e^{-V(x)} \dd x.
\end{align*}
As $p(m) = h$, the following automatically holds:
\[
H(m | m_*) \geqslant \inf_{\substack{ \nu\in \mathcal P_1(\R)\\p(\nu)=h}}
H(\nu|m_*).
\]
The infimum of the optimization problem
is attained, thanks to calculus of variations,
by some measure $\nu$ whose density satisfies
\[
- \log \nu(x) + V(x) - J\ell x = \text{constant},
\]
$\ell \in \R$ being the Lagrange multiplier.
Namely $\nu = \pi[\ell]$ with $p(\nu) = h$.
It follows that $f(\ell) = h$ and $\ell = f^{-1}(h)$.
Thus,
\[
H(m | m_*)
\geqslant H(\pi[\ell] | m_*)
= J\ell h - \log \int_{\R} e^{J\ell x - V(x)} \dd x
+ \log \int_{\R} e^{-V(x)} \dd x.
\]
Let $\varepsilon \in (0,1)$.
Observe that
\begin{IEEEeqnarray*}{rCl}
\IEEEeqnarraymulticol{3}{l}
{(1-\varepsilon) H(m|m_*) - Jh^2
+ \log \int_{\R} e^{Jhx - V(x)} \dd x
- \log \int_{\R} e^{-V(x)} \dd x} \\
\quad&\geqslant&
(1-\varepsilon)J\ell h
- (1-\varepsilon) \log \int_{\R} e^{J\ell x - V(x)} \dd x
- Jh^2
+ \log \int_{\R} e^{Jhx - V(x)} \dd x \\
&&\adjustbin- \varepsilon \log \int_{\R} e^{-V(x)} \dd x \\
&\eqqcolon& \varphi(h).
\end{IEEEeqnarray*}
The function $\varphi$ is even with $\varphi(0) = 0$ and has derivative
\[
\varphi'(h) =
J\bigl((1-\varepsilon) \ell
- 2h
+ f(h)\bigr),
\]
where we used the fact that $\ell' = 1/f'(\ell)$.
The second-order derivative reads
\[
\varphi''(h)
= J \biggl( \frac{1-\varepsilon}{f'(\ell)}
+ f'(h) - 2\biggr)
\]
We wish to show that $\varphi$ is positive everywhere
for some $\varepsilon$ small enough.
The function $\varphi$ being even,
we investigate it only on the positive half-line.
As $f' < 1$ and $f'' \leqslant 0$,
the equality $f(\ell) = h$ forces $\ell > h$,
and consequently $f'(\ell) \leqslant f'(h)$.
Hence $\varphi''$ satisfies
\[
\varphi''(h) \geqslant
J \biggl( \frac{1-\varepsilon}{f'(h)} + f'(h) - 2\biggr).
\]
Now take $\varepsilon = (1 - J/\Jc)^2$.
Note that the function
\[
t \mapsto \frac{1 - (1-J/\Jc)^2}{t} + t
\]
is decreasing for $t \in (0,J/\Jc]$.
Since $f'(h) \in (0, J/\Jc]$,
\[
\frac{1-(1-J/\Jc)^2}{f'(h)} + f'(h)
\geqslant \frac{1 - (1-J/\Jc)^2}{J/\Jc} + \frac{J}{\Jc} = 2.
\]
Therefore $\varphi''(h) \geqslant 0$ for all $h > 0$,
implying the positivity of $\varphi$.
In consequence,
\[
I(m | \Pi[m]) \geqslant 2\rho_0 H(m|\Pi[m])
\geqslant 2\rho_0\bigl(\varepsilon H(m|m_*) + \varphi(h)\bigr)
\geqslant 2 \rho_0 \varepsilon H(m|m_*).
\]
The non-linear log-Sobolev thus holds with
\[
\rho = \biggl(1 - \frac{J}{\Jc}\biggr)^{\!2}\rho_0.
\]

\proofstep{Step 3: Concentration of measure}
In this step we follow the strategy
in the second case of Proposition~\ref{prop:exm}.
Specifically, we consider a parameterized mean field interaction
and show the concentration for the corresponding particle
Gibbs measures by verifying the conditions of Theorem~\ref{thm:defective-t2}.

Let $\alpha \in [0,1]$ and $m_0 \in \mathcal P_1(\R)$.
Define the energy functional
\[
F_{\alpha,m_0}(m) = \langle V, m\rangle
- \alpha J p(m_0) p(m)
- \frac{(1-\alpha)J}{2} p(m)^2
\]
parameterized by $\alpha$, $m_0$,
where, as we recall, $p(m) = \int_{\R} x m(\dd x)$.
The local equilibrium mapping associated to
$F_{\alpha,m_0}$ reads
\[
\Pi_{\alpha,m_0} [m] = \Pi [\alpha m_0 + (1-\alpha) m].
\]
Denote $h_0 = p(m_0)$.
The fixed point relation
\[
\Pi_{\alpha,m_0} [m_{*,\alpha,m_0}] = m_{*,\alpha,m_0}
\]
is thus equivalent to $m_{*,\alpha,m_0} = \pi [h_{*,\alpha,m_0}]$
for some $h_{*,\alpha,m_0} \in \R$ solving
\[
h_{*,\alpha,m_0} = f \bigl( \alpha h_0 + (1-\alpha) h_{*,\alpha,m_0}\bigr).
\]
As $\lVert f' \rVert_{L^\infty} \leqslant J/\Jc < 1$,
the equation admits a unique solution by the contraction mapping theorem.

We proceed to show the non-linear Talagrand inequality
\eqref{eq:mf-neg-t2} for the measure $m_{*,\alpha,m_0}$ with the interaction
\[
(x,y) \mapsto -(1-\alpha)Jxy.
\]
By calculus of variations, the infimum of the mapping
\[
\nu \mapsto - \frac{\Jc}{2}
\bigl(p(\nu) - p(m_{*,\alpha,m_0})\bigr)^2 + H(\nu | m_{*,\alpha,m_0})
\]
is attained by some $\nu$ whose density solves
\[
- \log \nu(x) + V(x) - J\ell x = \text{constant}
\]
for some $\ell \in \R$.
In other words, $\nu = \pi[\ell]$ and in this case,
\begin{multline*}
- \frac{\Jc}{2}
\bigl(p(\nu)-p(m_{*,\alpha,m_0})\bigr)^2 + H(\nu | m_{*,\alpha,m_0}) \\
= - \frac{\Jc}{2}
\bigl(f(\ell) - f(h_{*,\alpha,m_0})\bigr)^2
+ H(\pi[\ell]|\pi[h_{*,\alpha,m_0}])
\eqqcolon \psi(\ell).
\end{multline*}
For simplicity of notation, we write
$h_* = h_{*,\alpha,m_0}$.
The function $\psi$ thus reads
\begin{IEEEeqnarray*}{rCl}
\psi(\ell)
&=& - \frac{\Jc}{2}
\bigl(f(\ell) - f(h_{*})\bigr)^2
+ J (\ell - h_*) f(\ell) \\
&&\adjustbin - \log \int_{\R} e^{J\ell x -V(x)} \dd x
+ \log \int_{\R} e^{Jh_* x -V(x)} \dd x.
\end{IEEEeqnarray*}
It satisfies $\psi(h_*) = 0$ and has derivative
\[
\psi'(\ell)
= \Bigl( - \Jc
\bigl( f(\ell) - f(h_*)\bigr)
+ J(\ell - h_*) \Bigr)f'(\ell).
\]
But we have
\[
\frac{\Jc}{J}\frac{f(\ell) - f(h_*)}{\ell - h_*}
\leqslant \frac{\Jc}{J} \sup f' \leqslant 1
\]
for all $\ell \neq h$.
Hence $\psi'(\ell) \geqslant 0$ for $\ell > h_*$
and $\psi'(\ell) \leqslant 0$ for $\ell < h_*$,
showing that $\psi(\ell) \geqslant 0$ for all $\ell \in \R$.
As a result, for all $\nu \in \mathcal P_1(\R)$,
\[
- \frac{\Jc}{2}
\bigl(p(\nu)-p(m_{*,\alpha,m_0})\bigr)^2 + H(\nu | m_{*,\alpha,m_0})
\geqslant 0.
\]
By rearranging the coefficients, we obtain
\begin{multline*}
- \frac{(1-\alpha)(\Jc+J)}{4}
\bigl(p(\nu)-p(m_{*,\alpha,m_0})\bigr)^2 + H(\nu | m_{*,\alpha,m_0}) \\
\geqslant
\frac 12 \biggl(1 - \frac{J}{\Jc}\biggr)H(\nu | m_{*,\alpha,m_0})
\geqslant
\biggl(1 - \frac{J}{\Jc}\biggr) \rho_0 W_2^2(\nu, m_{*,\alpha,m_0}),
\end{multline*}
the last inequality being a consequence of the $\rho_0$-log-Sobolev
for images of $\Pi$.
The non-linear Talagrand inequality \eqref{eq:mf-neg-t2} is now verified with
\[
\varepsilon = \frac 12
\min\biggl(\frac{\Jc}{J}-1,1\biggr),\qquad
\lambda = \biggl(1 - \frac{J}{\Jc}\biggr)\frac{\rho_0}{2}.
\]
We specialize to the case $\alpha = (k-1)/N$ and $m_0 = \mu_{\x^{[k-1]}}$.
Repeating the argument
in the second case of the proof of Proposition~\ref{prop:exm}
yields the transport and integrability constants $\gamma$, $M$
for Theorem~\ref{thm:sharp}.
\qed

\subsubsection*{Acknowledgements}
S.\,W. wishes to express his thanks to Benoît Dagallier
for discussions on his work~\cite{BBDCriterionFreeEnergy}
prior to its preprint posting
and to Pierre Monmarché for discussions on his recent work
\cite{MonmarcheReygnerLocal}.
The authors are also grateful to Daniel Lacker
for identifying an error in Remark~\ref{rem:curie-weiss-critical}
in an earlier version of this paper.

\subsubsection*{Funding}
Z.\,R's research is supported by
the Finance For Energy Market Research Centre
and by the France 2030 grant (ANR-21-EXES-0003).

\printbibliography
\end{document}